\documentclass[final]{article}

\usepackage{amsmath,amsthm,amssymb,amsfonts}
\usepackage{indentfirst}
\usepackage[notcite,notref]{showkeys}
\usepackage{enumitem}

\newcommand{\R}{\mathbb{R}}

\renewcommand{\S}{{\mathbb{S}^{N-1}}}
\newcommand{\un}{\mathbf{1}}
\newcommand{\eps}{\varepsilon}

\newcommand{\D}{\mathcal{D}}

\newtheorem{defi}{Definition}
\newtheorem{lem}{Lemma}
\newtheorem{prop}{Proposition}
\newtheorem{thm}{Theorem}

\newtheorem{assumption}{Assumption}

\newtheorem{ex}{Example}
\theoremstyle{remark}
\newtheorem{rem}{Remark}

\title{ Phasefield theory \\
for fractional diffusion-reaction equations \\
and applications}

\author{Cyril Imbert\footnote{Universit\'e Paris-Dauphine, CEREMADE, UMR CNRS 7534,
    place de Lattre de Tassigny, 75775 Paris cedex 16, France, \texttt{imbert@ceremade.dauphine.fr}}
and Panagiotis E. Souganidis\footnote{University of Chicago, Dept. of Mathematics,
5734 S. University Avenue Chicago, Illinois 60637, USA, \texttt{souganidis@math.uchicago.edu}}
}

\date{\today}

\begin{document}
\maketitle

\begin{quote} \footnotesize
  \noindent \textsc{Abstract.}  This paper is concerned with
  diffusion-reaction equations where the classical diffusion term,
  such as the Laplacian operator, is replaced with a singular
  integral term, such as the fractional Laplacian operator.  As far as
  the reaction term is concerned, we consider bistable
  non-linearities. After properly rescaling (in time and space) these
  integro-differential evolution equations, we show that the limits of
  their solutions as the scaling parameter goes to zero exhibit
  interfaces moving by anisotropic mean curvature. The singularity and
  the unbounded support of the potential at stake are both the novelty
  and the challenging difficulty of this work.
\end{quote}


49L25, 53C44

\section{Introduction}

This paper is concerned with diffusion-reaction equations where the
classical diffusion is replaced with a singular integral term.
Our aim is somewhat classical: to show that the limit of their
solutions after properly rescaling them in time and space exhibit 
a moving interface. However, we will deal with integral term
whose potential is anisotropic, singular and with unbounded support. 

\paragraph{Fractional diffusion-reaction equations.} We 
consider for $t>0$ and $x\in \R^N$, $N\ge 2$,
\begin{equation}
\label{eq:rd>1}
\partial_t u^\eps + \frac1{\eps^2} \bigg\{ -\mathcal{I}_\alpha^\eps u^\eps 
+  f (u^\eps) \bigg\} =0 
\end{equation}
and
\begin{equation}
\label{eq:rd=1}
\partial_t u^\eps + \frac1{\eps^2 |\ln \eps|} \bigg\{ -\mathcal{I}_1^\eps u^\eps 
+  f (u^\eps) \bigg\} =0 
\end{equation}
where $\mathcal{I}_\alpha^\eps$ is a singular integral operator
depending on a parameter $\alpha \in (0,2)$ and $f$ is a bistable
non-linearity.
\begin{ex}[Standing example]
The standing example for $f$ and $\mathcal{I}^\eps_\alpha$ are
$$
f(u) = u (u^2-1) \quad \text{and} \quad \mathcal{I}_\alpha^\eps u = \frac1{\eps^\alpha} \Delta^{{\frac\alpha{2}}} u 
$$
where $\Delta^{{\frac\alpha{2}}}$ denotes the fractional Laplacian of the function $u$. We recall
that 
$$
\Delta^{{\frac\alpha{2}}} u = \int ( u (x+z) -u (x)) \frac{dz}{|z|^{N+\alpha}}
$$
where the singular integral must be undertood in the sense of Cauchy's principal value. 
\end{ex}
More generally, we will consider singular integral operators of the following form
\begin{equation} \label{eq:op}
\mathcal{I}_\alpha^\eps u (x) = 
\int \bigg(u(x+\eps z) - u (x) - \eps D u(x) \cdot z \un_B ( \eps z) \bigg) J(z) dz 
\end{equation}
where $B$ denotes the unit ball and where the function $J : \R^N \to \R$, which will be often referred to 
as the \emph{potential}, can be of two types
\begin{equation}\label{def:J}
\text{either } \qquad J (z) = g \left(\frac{z}{|z|} \right) \frac1{|z|^{N+\alpha}} \qquad
\text{ or } \qquad J \in L^1 (\R^N) \cap C_c (\R^N)
\end{equation}
with  $\alpha \in [1,2)$ and $g : \S = \{ z \in \R^N: |z|=1 \} \to (0,+\infty)$ continuous
and where $C_c (\R^N)$ denotes the space of continuous functions with bounded support. 
The first potential will be referred to as the singular one while the second one
will be referred to the regular one. As far as the standing example is concerned, $J$ is
singular with $g \equiv 1$. 

\paragraph{Phasefield theory for diffusion-reaction equations.}  
In \cite{chen}, Chen proved rigourously that the solution of the
Allen-Cahn equation \cite{ac} generates a front moving with mean
curvature as long as the front is regular. Thanks to definition of
fronts past singularities \cite{es,cgg}, Evans, Soner and the second author
\cite{ess92} proved that this is still true after the appearence of
singularities. Such results are generalized to a large class of
bistable non-linearities by Barles, Soner and the second author \cite{bss93}
where a general phasefield theory for reaction-diffusion equations is
introduced. In \cite{bs98,bd03}, an abstract method is developed in
order to deal with more general reaction-diffusion equations and to handle
boundary conditions. In particular, non-local reaction-diffusion
equations are considered in \cite{bs98} but integral operators are not
singular. As the proofs of the present paper will show it, it is a
challenging difficulty to be overcome.

\paragraph{Motivations.} Recently, Caffarelli and the second author studied
threshold dynamics-type algorithms corresponding to the fractional
Laplace operator for $\alpha \in (0,2)$. They proved that after
properly rescaling them, they converge to an interface moving by mean
curvature in the case $\alpha \ge 1$ and to a fractional mean
curvature in the case $\alpha <1$. Hypersurfaces with zero integral
curvature are studied in \cite{crs}. See also \cite{imbert09} where
the level-set approach \cite{es,cgg} is developed for such a geometric flow.

As far as applications are concerned, two main physical models
motivate the present study. The first application we have in mind is
dislocation dynamics. Dislocation theory aims at explaining the
plastic behaviour of materials by the motion of linear defects in
crystals. Peirls-Nabarro models \cite{kco} consist in approximating
the geometric motion of these defects by non-local diffusion-reaction
equations such as \eqref{eq:rd=1}. In \cite{fm}, such an approximation
is also used and formal expansions are performed. In \cite{gm05,gm06},
Garroni and M\"uller study a variational model for dislocations that
can be viewed as the variational formulation of the stationary version
of \eqref{eq:rd=1}.

The second application we have in mind is statistical mechanics and
more precisely stochastic Ising models. These models were introduced
by Ka\u{c}, Uhlenbeck and Hemmer \cite{kuh} (see also \cite{lp}) to
justify the validity of the Van der Waal's phase diagram. The
interaction beween particles is described by the Ka\u{c} potential. A
lot of work has been done since then to understand the hydrodynamic
limits of such interacting particle systems and it is beyond the scope
of this paper to give a complete list of references. However, we can
mention the papers by De Masi, Orlandi, Presutti and Triolo
\cite{dopt93,dopt94a,dopt94b,dopt96a,dopt96b} and Katsoulakis and the
second author \cite{ks94,ks95,ks97}. The interested reader is also
referred to the monograph of De Masi and Presutti \cite{dmp} and the
book of Spohn \cite{spohn}. In the papers we mentioned before,
hydrodynamic limit of stochastic Ising models with general dynamics
are studied. In particular, a mean field equation is derived, see
\cite{dopt94a}. By many ways, these equations can be viewed as
non-local reaction-diffusion equations. The next step is to show that
for appropriate scalings the solution of the mean field equation
approximates an anisotropic mean curvature motion; see for instance
\cite{dopt93,ks94}. Green-Kubo type formulae are provided for the
mobility and the diffusion matrix in terms of a standing wave
associated with the mean field equation.

On one hand, the Ka\u{c} potential is
assumed, in most papers, to be regular with a compact support.  On the
other hand, Lebowitz and Penrose \cite[Eq.(1.20b),(1.21a),p.100]{lp}
consider potentials $J$ that are singular; more precisely, they assume
that for small $z$, the singularity of $J$ is of the form
$|z|^{-N-\alpha}$ for $\alpha >0$.

\paragraph{Description of the results.} Our main result states that,
as $\eps \to 0$, the solutions $u^\eps$ of \eqref{eq:rd>1} and
\eqref{eq:rd=1} can only have two limits: the stable equilibria of the
bistable non-linearity $f$ (see Section~\ref{sec:twmf} for
definitions). The resulting interface evolves by anisotropic mean
curvature; moreover, Green-Kubo-type formulae are obtained: the
mobility and the diffusion matrix of the geometric flow are expressed
in terms of the standing wave associated with the bi-stable
non-linearity; see Eq.~\eqref{eq:mobility-rd}, \eqref{eq:matrix-rd>1}
and \eqref{eq:matrix-rd=1} below. Even if the proof follows the
classical idea of constructing barriers by using traveling waves, the
reader will see that classical arguments fail when extending the
barrier away from the front; several new ideas are needed to handle
the unboundedness of the support.  In order to handle the anisotropy
of the potential, we have to use ideas developed by Katsoulakis and
the second author \cite{ks95} and introduce correctors to cancel
oscillating terms by averaging them.  This implies in particular that
anisotropic traveling waves must be considered. But because the
integral term involves a singular potential, passing to the limit
in averaged oscillating terms is challenging and this constitutes the
core of the proof of the convergence theorem.

As the reader will see it when going through the preliminary section
or in the statement of the convergence theorem, several assumptions on
traveling waves and the linearized traveling wave equation are
necessary (if not mandatory). Even if we do not construct such waves
and correctors and assume that they exist, the reader can check that the
assumptions we make are natural. For instance, the decay
estimate~\eqref{eq:q-decay} is expected since its corresponds to the
one of the kernel of the fractional Laplacian in the one dimensional space. See
also \cite{cs}. We plan to construct them in a compagnion paper.

As explained above, we will consider two kinds of potentials: singular
and regular ones. As far as the singular case is concerned, we
distinguish two subcases, depending how singular is the potential at
the origin. Since potentials are positively homogeneous in the
singular case, potentials in the subcase $\alpha=1$ decay as
$|z|^{-N-1}$ when $|z| \to + \infty$. This corresponds to the
dislocation dynamics model. As the reader can see it, the scaling
involves a logarithmic term; this factor is well-known in physics and
the interested reader is referred to \cite{brown} for instance; see
also \cite{dfm,cs}.  An additional comment about singular and regular
potentials concern the Green-Kubo-type formulae. It turns out that
these formulae are different in singular and regular
cases. However, we give in appendix a formal argument to shed some
light on the link between these two formulae.

\paragraph{Additional comments.} As the reader can see it, we are not
able to deal with the case $\alpha <1$ even if, in view of the results
of \cite{cs}, we should observe an interface moving with fractional
mean curvature (see Section~\ref{sec:twmf} for a definition). In
this case, the equation should be rescaled in time as follows
\begin{equation}
\label{eq:rd<1}
\partial_t u^\eps + \frac1{\eps^{1+\alpha}} \bigg\{ -\mathcal{I}_\alpha^\eps u^\eps 
+  f (u^\eps) \bigg\} =0 
\end{equation}
with
\begin{equation*} 
\mathcal{I}_\alpha^\eps u (x) = 
\int \bigg(u(x+\eps z) - u (x)  \bigg) g \left( \frac{z}{|z|} \right) \frac{dz}{|z|^{N+\alpha}} 
\end{equation*}
for some $\alpha \in (0,1)$.  The reader can check that we are able to
pass to the limit in (the average of) oscillating terms (see
Lemma~\ref{lem:a-eps-conv-2} below), which is usually the difficult
part of the convergence proof.  We are even able to construct a
barrier close to the front. But because the diffusion-reaction is
non-local, we are stuck with extending the solution away from it. In
particular, the very slow decay of the potential at infinity does not
permit us to use the new ideas we introduced in the singular case $\alpha \ge 1$. 
This difficulty is unexpected since in \cite{cs}, this case is the easiest one. 
We hope to find a path toward this result in a future work.

In the one dimensional space, moving interfaces are points. 
Gonzalez and Monneau \cite{gm} considered such a case and
proved a result analogous to our main one by taking advantage
of the fact that the limit is a (system of) ordinary differential
equation(s). In particular, the restriction on the strength of
the singularity can be relaxed in this case. 

\paragraph{Organization of the article.} The first section is devoted
to preliminaries. In particular, traveling waves are introduced as
well as the linearized traveling wave equation which is the equation
satisfied by correctors; see Subsections~\ref{subsec:tw} and
\ref{subsec:ltwe}. We also introduce the geometric motion by mean
curvature (Subsection~\ref{subsec:geom}) together with its equivalent
definition in terms of generalized flows (Subsection~\ref{subsec:gen
  flow}).  Our main result is stated in Section~\ref{sec:frde}. In the
remaining of this section, we explain how to reduce the proof of this
convergence result to the construction of an appropriate barrier (see
above). Section~\ref{sec:construction} is dedicated to this
construction. The last section (Section~\ref{sec:lem}) contains to
core of the proof of the convergence result: the limit of the average
of oscillating terms.  Finally, we give in appendix a formal argument
to explain the link between the two Green-Kubo formulae obtained in
the convergence theorem.

\paragraph{Notation.} The Euclidian norm of $x \in \R^N$ is denoted by
$|x|$. The ball of center $x$ and of radius $r$ is denoted by $B_r
(x)$.  We simply write $B_r$ for $B_r(0)$ and $B=B_1$ denotes the unit
ball.  The scalar product of $x$ and $y$ is denoted by $x \cdot y$.
The unit sphere of $\R^N$ is denoted by $\S$.  The set of symmetric $N
\times N$ matrices is denoted by $\mathcal{S}_N$.  The identity matrix
(in any dimension) is denoted by $I$.

Given two real numbers $a,b \in \R$, $a \vee b$ denotes $\max(a,b)$
and $a \wedge b$ denotes $\min (a,b)$, $a^+$ denotes $a \vee 0$ and
$a^-$ denotes $- (a \wedge 0)$.  In particular, $a^\pm \ge 0$ and $a =
a^+ - a^-$.

Given a set $A$, $\un_A$ denotes its indicator function that equals
$1$ in $A$ and $0$ outside. The signed distance function $d_A (x)$
associated with $A$ equals the distance function to $\R^N \setminus A$
if $x \in A$ and the opposite of the distance function to $A$ if $x
\notin A$.

The set of continuous functions $f : \R^N \to \R$ with compact support
is denoted by $C^0_c$.

Given a family of locally bounded functions $f_\eps : \Omega \subset
\R^d \to \R$ indexed by $\eps >0$, the relaxed upper and lower limits
are defined as follows
$$
\liminf{}_* f_\eps (x) = \liminf_{\eps \to 0, y \to x} f_\eps (y)
\quad \text{ and } \quad \limsup{}_* f_\eps (x) = \limsup_{\eps \to 0,
  y \to x} f_\eps (y) \; .
$$
If $f_\eps = f$ for any $\eps>0$, these relaxed semi-limits coincide
with the lower and upper semi-continuous of a locally bounded function
$f$.

For traveling waves $q(r,e)$ and correctors $Q(r,e)$, $\dot{q}$ and
$\dot{Q}$ denote derivatives with respect to $r$.

\section{Preliminaries}
\label{sec:twmf}

This section is devoted to the presentation of the assumptions we make about non-linearities, 
traveling waves
and  linearized traveling wave equations. We also briefly describe the construction of fronts
at stake after rescaling fractional diffusion-reaction equations. 

\subsection{Fractional diffusion-reaction equations}

We can write \eqref{eq:rd>1}, \eqref{eq:rd=1} and \eqref{eq:rd<1} as follows
\begin{equation}
\label{eq:rd}
\partial_t u^\eps + \frac1{\eps \eta} \bigg\{ -\mathcal{I}_\alpha^\eps u^\eps 
+  f (u^\eps) \bigg\} =0 
\end{equation}
with 
\begin{equation}\label{def:eta}
\eta = 
\left\{ 
  \begin{array}{ll} 
    \eps & \text{ if } \alpha >1 \, , \\ 
    \eps |\ln \eps | & \text{ if } \alpha =1 \, , \\ 
    \eps^\alpha & \text{ if } \alpha < 1  \\ 
  \end{array}
\right.
\end{equation}
(see Remark~\ref{rem:choice-h}).

We will use later on that 
the potential $J$ satisfies in the singular case the following properties
\begin{equation} \label{cond1:J}
\left\{ 
\begin{array}{l}
J \text{ is smooth on $\R^N \setminus \{0\}$, even and non-negative} \medskip \\
|J (z)| \le \frac{C_J}{|z|^{N+\alpha}} \\
J (z) \sim g\left(\frac{z}{|z|} \right) \frac{1}{|z|^{N+\alpha}} \text{ as } |z| \to +\infty 
\end{array}
\right.
\end{equation}
with $\alpha \in (0,2)$. We also mention that if $\alpha < 1$, then
\begin{equation} \label{cond2:J}
|z|^2 J_\eps (z) + |\nabla ( |z|^2J_\eps (z) )| \le K (z) \in L^1 (B) 
\end{equation}  
where $J_\eps (z) = \eps^{-N - \alpha} J (\eps^{-1} z)$. 

\subsection{Bistable non-linearity}

We make the following assumptions. 
\begin{assumption}[Bistable non-linearity]\label{assum:fbistable}
The non-linearity $f:\R \to \R$ is $C^1$ and such that
\begin{itemize}
\item
for all $h \in (0, H)$, they are constants $m_\pm(h)$ and $m_0(h)$ such that
\begin{equation} \label{eq:zero-f}
f(m_i(h) ) = h  , \quad i \in \{0,+,-\} \qquad \text{ and } \quad 
m_- (h) < m_0(h) < m_+(h).
\end{equation}
\item
$f>0$ in $(\bar{m}_-,\bar{m}_0)$ and $f<0$ in $(\bar{m}_0,\bar{m}_+)$.
\item
$f'(\bar{m}^\pm) > 0$ and $f'(\bar{m}_0) < 0$ where $\bar{m}_\pm = m_\pm(0)$ and $\bar{m}_0 =m_0(0)$. 
\end{itemize}
In particular, if $h >0$,  we have $\bar{m}_\pm \le m_\pm(h)$ and $\bar{m}_0 \ge m_0(h)$  and 
\begin{equation}\label{eq:m-eps-bar}
|m_\pm (h)- \bar{m}_\pm| \le C_f  h \; .
\end{equation}
\end{assumption}

\subsection{Anisotropic traveling wave}
\label{subsec:tw}

In this subsection, we describe the anisotropic traveling waves
we will use in the construction of barriers in order to get the main
convergence result. In particular, we make precise the decay we
expect for such waves. This construction will be achieved in a future work. 
\begin{assumption}[Anisotropic traveling wave]\label{assum:atw}
For $h \in (0,H)$, there then exist two continuous functions 
$q  : \R \times \S \to \R$ and $c : \R \times \S \to \R$ such that
$$
q (r,e,h) \to m_\pm (h) \text{ as } r \to \pm \infty
$$
(the limit being uniform with respect to $e \in \S$)
and
\begin{equation} \label{eq:tw-rd}
c \dot{q} - \mathcal{I}_{e} \left[q \right] + f(q) = h 
\end{equation}
where
$$
\mathcal{I}_{e} [q] \; (\xi) = \left\{ \begin{array}{ll}
 \int \bigg( q (\xi +{e}\cdot z) - q(\xi) 
- \dot{q} (\xi){e}\cdot z  \un_B (z) \bigg) J(z) dz & \text{ if } \alpha \ge 1 \\
 \int \bigg( q (\xi +{e}\cdot z) - q(\xi)  
\bigg) J(z) dz & \text{ if } \alpha < 1 
\end{array}\right.
$$
for any $e \in \S$.

The traveling wave $q$ is increasing in $r$ and the following estimates hold true
for any $(r,{e}) \in \R \times \S$
\begin{eqnarray} 
\label{eq:q-decay}
 |q(r,e,h) - m_\pm(h)| &=& O\left(\frac1{|r|^{1+\alpha}} \right) \quad \text{ as } r \to \pm \infty \\
\label{eq:q-estim}
\sup_{h>0} \{ \| D_{e} q \|_\infty + \| D^2_{e,e}  q \|_\infty \} &\le& C_q 
\end{eqnarray}
for some constant $C>0$ and with a limit uniform in $e \in \S$. 
The function $q$ also satisfies
\begin{equation}\label{eq:dotq-l2}
\int \dot{q}^2(\xi) d\xi \to \int (\dot{q^0})^2 (\xi) d\xi
\end{equation}
as $h \to 0$ and the limit is uniform in $e$. 

The speed $c$ satisfies
\begin{equation}\label{eq:cbar-monotone}
h c (e,h) > 0 \quad \text{ if } \quad h \neq 0
\end{equation}
and
\begin{equation} \label{eq:cbar}
\frac{c(e,h)}{h} \to \overline{c}({e}) \quad \text{ as } \quad h \to 0
\end{equation}
and the previous limit is uniform with respect to $e$. Moreover,
the function $\overline{c}({e})$ is continuous on $\S$.
\end{assumption}

\paragraph{Standing wave.}
Notice that \eqref{eq:cbar} implies in particular that $c(e,0)=0$. 
Hence $q(e,0)$ is a standing wave. It is denoted by $q^0$ in the remaing
of the paper.

\paragraph{Reduced integral operator.}
The operator $\mathcal{I}_{e}$ does not depend on $e$ if $J (z)$ is
radially symmetric, \textit{i.e.} when $J(z) = j (|z|)$. We illustrate
this fact in the next lemma where $\mathcal{I}_e$ is computed in the
case where $J (z) = g(\hat{z}) |z|^{-N-\alpha}$.
\begin{lem}
Assume that $J (z) = g (\frac{z}{|z|}) \frac1{|z|^{N+\alpha}}$. 
Then
$$
\mathcal{I}_e [q] (r) =  a_{11} (e) \int_\R (q(\xi + r) - q(\xi)) \frac{d\xi}{|\xi|^{1+\alpha}}   
$$
where 
$$
a_{11} (e) = \int_{\R^{N-1}} g (1,u) \frac{du}{(1+|u|^2)^{(N+\alpha)/2}}
$$
\end{lem}
\begin{rem}
We recognize the fractional Laplacian of order $\alpha$ in the one dimensional space (up 
to a multiplicative constant).
\end{rem}

\subsection{Linearized traveling wave equation}
\label{subsec:ltwe}

In this subsection, the linearized traveling wave equation is
considered.  Loosely speaking, we need to know that the kernel of the
linearized operator $\mathcal{L}$ reduces to $\R \dot{q}$ and, if $f$
is regular enough, so is the solution $Q$ of $\mathcal{L} Q (\xi)= Pf
(\xi,t,x)$ where $Pf$ is the projection of $f$ on the space orthogonal
to $\dot{q}$.  We need in particular to be able to say that $Q$ decays
at infinity. Let us be more precise now.  \medskip

The linearized  operator $\mathcal{L}$ associated with \eqref{eq:tw-rd} around
a solution $q$ is
\begin{equation} \label{eq:linearized-rd}
\mathcal{L} Q = c \dot{Q} - \mathcal{I}_{e} [Q] + f'\left(q \right) Q.
\end{equation}
Given a smooth function $d:[0,T] \times \R^N \to \R$, let us consider
$$
a (r,e,t,x) = \frac1{h} \int \bigg\{  q (r+ \frac{d(t,x+\eps z) - d(t,x)}\eps  , e ) -  q ( r + e \cdot z, e)  \bigg\} J(z) dz 
$$
where $h = \eps$ (resp. $\eps |\ln \eps|$, $\eps^\alpha$) if $\alpha >1$ (resp. $\alpha =1$, $\alpha <1$).
\begin{assumption}[The linearized TW equation]\label{assum:Q}
$$
\mathrm{Ker} \; \mathcal{L} = \mathrm{Ker} \; \left( \mathcal{L} \right)^* 
= \mathrm{span} \; \dot{q}. 
$$
Moreover, if $d: [0,T] \times \R^N \to \R$ is a smooth function, 
there then exists a continuous solution $Q: \R \times \S \times [0,T] \times \R^N$ to 
$$
\mathcal{L} Q = P_{\dot{q}} a(\cdot,e,t,x)
$$
where $P_{\dot{q}}$ stands for the projection on the space orthogonal to $\mathrm{span} \; \dot{q}$.
In particular, there exists $C_Q>0$ such that for any $h,e,\xi$,
\begin{eqnarray}
\label{eq:Q-estim}
|\dot{Q}|+ |\partial_t Q|+|D_x Q| + |D_{e} Q| + | D^2_{e,e} Q | + | D^2_{x,x} Q |\le C_Q
\\
\label{eq:Q-decay}
Q  \to 0  \quad \text{ as } r \to +\infty 
\end{eqnarray}
where $C_Q$ does not depend on $h,e,t,x$ and the limit is uniform in $h,e,t,x$.  
\end{assumption}

\subsection{Geometric motions}
\label{subsec:geom}

In this subsection, we introduce the geometric motions of fronts at
stake when rescaling the fractional diffusion-reaction equations.

It is well-known that singularities can appear on the front in finite
time when considering, for instance, the mean curvature motion.  We
thus classically use the level-set approach to define a front for all
times.  We recall that this approach consists in looking for a front
$\Gamma_t$ under the form $\{ x : u (t,x) =0\}$ and to derive a PDE
satisfied by $u$.  \medskip

\paragraph{Anisotropic mean curvature motion.}
In the case of an anisotropic mean curvature motion, we obtain the following degenerate and singular
parabolic equation,
\begin{equation} \label{eq:amcm}
  \partial_t u = \mu \left( \widehat{D u} \right) 
  \mathrm{Tr} \left((I -{ \widehat{D u} }\otimes{ \widehat{D u} }) A \left( \widehat{D u} \right) D^2 u \right) 
\end{equation} 
where $\mu: \S \to \R^+$ and $A:\S \to \mathcal{S}_N$ are continuous
functions and $I$ stands for the $N\times N$ identity matrix and $e =
\frac{p}{|p|}$.  We will see that the function $\mu$ (which will be referred
to as the mobibility) is given by the following
formula
\begin{equation} \label{eq:mobility-rd}
\mu (e) = \left\{ \int (\dot{q^0})^2 (\xi,e) d\xi\right\}^{-1}  \, .
\end{equation}
As far as the function $A$ is concerned, we distinguish cases. 
In the singular case and if $\alpha >1$, we have for all $e \in \S$
\begin{equation}
\label{eq:matrix-rd>1}
A (e)  =  \bigg\{ \int \int (q^0 (\xi + z_1,e) -q^0 (\xi,e))^2 \frac{dz_1}{|z_1|^{1+\alpha}}   d \xi \bigg\} A_g (e)
\end{equation}
with
\begin{equation}\label{def:ag}
A_g (e) =\alpha (\alpha -1)\int_{\R^{N-1}}  (1,u) \otimes (1,u) g(1,u) \frac{du}{(1+ |u|^2)^{(N+\alpha)/2}} 
\end{equation}
(where the space orthogonal to $e$ is identified with $\R^{N-1}$).
If $\alpha =1$, 
\begin{equation}
\label{eq:matrix-rd=1}
(\bar{m}_+-\bar{m}_-)^2 \int_{\mathbb{S}^{N-2} = \{ z \in \S: z \cdot e =0\}}  
\bigg( \theta \otimes \theta J(\theta) \bigg) \sigma (d \theta) \, .
\end{equation}
In the regular case, we have for all $e \in \S$
\begin{equation}\label{eq:matrix-rd>1bis}
A (e)  =  \int \dot{q^0}(\xi,e) \dot{q^0} (\xi+e \cdot z) z \otimes z J(z) dz \, .
\end{equation}
\begin{rem}
  As a matter of fact, in the singular case with $\alpha >1$, $A$'s
  given by \eqref{eq:matrix-rd>1} and \eqref{eq:matrix-rd>1bis} are
  the same, at least formally.  But it is not even clear that the
  integral defining $A(e)$ in \eqref{eq:matrix-rd>1bis} is well
  defined in the case $\alpha > 1$. A formal argument is given in
  Appendix.
\end{rem}

\paragraph{Fractional mean curvature motion.}
A fractional version of this motion can be defined. More precisely, for any $\alpha <1$,
one can consider the following PDE, 
\begin{equation} \label{eq:mnl}
  \partial_t u =  \mu \left( \widehat{D u} \right) \kappa [x,u(t,\cdot)]
  |D u |
\end{equation} 
where $\mu$ is defined by \eqref{eq:mobility-rd} and 
\begin{multline*}
\kappa [x,U] = 
\kappa^* [x,U] = 
\int \bigg\{ \un_{\{U(x+z) \ge  U(t,x),  e \cdot z \le 0\}} \\
- \un_{\{U(t,x+z) < U(t,x), e \cdot z > 0 \}}\bigg\}
g\left(\frac{z}{|z|} \right) \frac{dz}{|z|^{N+\alpha}} \; . 
\end{multline*}
This can also be written under the general form
\begin{multline*}
\kappa^* [x,U] = 
\nu \{ z \in \R^N : U(x+z) \ge  U(t,x),  e \cdot z \le 0\} \\ - 
\nu \{ z \in \R^N : U(t,x+z) < U(t,x), e \cdot z > 0 \}
\end{multline*}
for some non-negative Borel measure $\nu$ which is eventually
singular. A general theory is developed in \cite{imbert09} to prove that the
geometric flow is well defined.  The definition of a viscosity
solution for \eqref{eq:mnl} implies the use of the following quantity
\begin{multline*}
\kappa_* [x,U] = 
\nu \{ z \in \R^N : U(x+z) >  U(t,x),  e \cdot z < 0\} \\ - 
\nu \{ z \in \R^N : U(t,x+z) \le U(t,x), e \cdot z \ge 0 \} \, .
\end{multline*}

\paragraph{Geometric non-linearities.}
In the following, we will use the notation: if $\alpha \ge 1$,
\begin{equation}\label{def:F}
F (p,X) = 
-  \mu (\hat{p})  \mathrm{Tr} \left( \tilde{A} ( \hat{p} ) X \right) 
-\overline{c} (0,\hat{p}) |p| 
\end{equation}
and if $\alpha <1$,
\begin{equation}\label{def:Fbis}
F (p,X,[\phi]) = -  \mu (p) \kappa^* [x,\phi] |p| -\overline{c} (0,\hat{p}) |p| 
\end{equation}
for $p \neq 0$ and $\hat{p} = p /|p|$. Since non-linearities $F$ are discontinuous,
it is necessary to use the lower and upper semi-continuous envelopes $F_*$ and $F^*$ 
of $F$ in order to define viscosity solutions of \eqref{eq:amcm}. In the case $\alpha < 1$, 
we have 
\begin{eqnarray*}
F^* (p,X,[\phi]) =  -  \mu_* (\hat{p}) {\kappa}_* [x,\phi] |p| -\overline{c}_* (0,\hat{p}) |p| \\
F_* (p,X,[\phi]) =  -  \mu^* (\hat{p}) {\kappa}^* [x,\phi] |p| -\overline{c}^* (0,\hat{p}) |p| 
\end{eqnarray*}
with the convention $\hat{0} = 0$. 

\subsection{Generalized flows}
\label{subsec:gen flow}

As we explained it above, the level-set approach is necessary in order 
to define the anisotropic mean curvature motion of a curvature after
the onset of singularities. It is proved in \cite{bs98} (see also \cite{bd03})
that this notion of solution is intimately related with the notion of
generalized flows of interfaces whose definition is recalled next.
\begin{defi}[Generalized flow for \eqref{eq:amcm}]\label{def:gen-flow}
  An family $\Omega = (\Omega_t)_{t>0}$ (resp. $U=(U_t)_{t>0}$) of open
  (resp. closed) sets of $\R^N$ is a \emph{generalized super-flow}
  (resp. \emph{generalized sub-flow}) of \eqref{eq:amcm} if for all
  $(t_0,x_0) \in (0,+\infty) \times \R^N$,  $h>0$, and for all
  smooth function $\phi : (0;+\infty) \times \R^N \to \R$ such that
\begin{enumerate}[label=(\roman{*})]
\item \label{0} \textbf{(Boundedness)} There exists $r>0$ such that
$$
\{ (t,x) \in [t_0,t_0+h] \times \R^N : \phi (t,x) \ge 0 \} \subset
[t_0,t_0+h] \times B (x_0,r),
$$
\item \label{i} \textbf{(Speed)} There exists $\delta_\phi>0$ such that
\begin{eqnarray*}
\partial_t \phi + F^* (D\phi,D^2\phi) \le -\delta_\phi \text{ in } [t_0,t_0+h]\times \bar{B}(x_0,r)
\\ 
\text{(resp. }  \partial_t \phi + F_* (D\phi,D^2\phi) \ge \delta_\phi \text{
  in } [t_0,t_0+h]\times \bar{B}(x_0,r) \text{),}
\end{eqnarray*}
\item \label{ii} \textbf{(Non-degeneracy)}
$$
D\phi \neq 0 \text{ in } \{ (s,y) \in [t_0,t_0+h] \times
\bar{B}(x_0,r) : \phi(s,y)=0 \},
$$
\item \label{iii}  \textbf{(Initial condition)}
\begin{eqnarray*}
  \{ y \in \bar{B}(x_0,r) : \phi (t_0,y) \ge 0 \} \subset \Omega_{t_0}, \\
  \text{(resp. } \{ y \in \bar{B}(x_0,r) : \phi (t_0,y) \le 0 \} \subset \R^N \setminus F_{t_0} \text{),}
\end{eqnarray*}
\end{enumerate}
then 
\begin{eqnarray*}
 \{ y \in \bar{B}(x_0,r) :  \phi(t_0+h,y) > 0\} \subset \Omega_{t_0+h} \\
\text{(resp. }  
\{ y \in \bar{B}(x_0,r) :  \phi(t_0+h,y) < 0\} \subset \R^N \setminus F_{t_0+h} \text{).} 
\end{eqnarray*}
\end{defi}
\begin{rem}
  Remark that this definition slightly differs from the one introduced
  in \cite{bs98}.  However, a quick look at the proof of the abstract
  method from \cite{bs98} that will be used below should convince the
  reader that this definition is adequate too.
\end{rem}

\section{The convergence result}
\label{sec:frde}

This section is devoted to statement of the main result of this paper.
We also explain why its proof reduces to the construction of an
appropriate ``barrier'' which will be constructed in the next section. 

Loosely speaking, we will prove that solutions of the fractional
diffusion-reaction equation~\eqref{eq:rd} approximate, as $\eps$ go to
$0$, the motion of a front moving with a normal speed equal to its
mean curvature.  Moreover, the results state that the mean curvature
motion is anisotropic and that mobilities and diffusion matrices are
given by Green-Kubo formulae (see \eqref{eq:mobility-rd},
\eqref{eq:matrix-rd>1} and \eqref{eq:matrix-rd=1}).

\subsection{Statement of the main result}

We now state our main result. 
\begin{thm}[Convergence result when $\alpha \ge
  1$] \label{thm:rd>1-convergence} Let $J$ be given by \eqref{def:J}
  with $\alpha \ge 1$ in the singular case and let $f$ be a bistable
  non-linearity. We suppose that Assumptions~\ref{assum:fbistable},
  \ref{assum:atw}, \ref{assum:Q} are satisfied by $f$.  Let $u^\eps$
  be the unique solution of \eqref{eq:rd>1} if $\alpha >1$ and of
  \eqref{eq:rd=1} if $\alpha=1$ associated with a continuous initial
  datum $u_0^\eps: \R^N \to [\bar{m}_-,\bar{m}_+]$ defined by
\begin{equation}\label{eq:ic}
  u_0^\eps (x) = q^0 \left( \frac{d_0 (x)}\eps, Dd_0 (x) \right)
\end{equation}
where $q^0$ is the standing wave associated with the
diffusion-reaction equation and $d_0$ is the signed distance function
to the boundary of a smooth set $\Omega_0$.

Let $u$ be the unique solution of the geometric
equation~\eqref{eq:amcm} supplemented with the initial condition
$u(0,x) = d_0 (x)$, where $\mu$ is given by \eqref{eq:mobility-rd} and
$A$ is defined in \eqref{eq:matrix-rd>1bis} in the regular case and
\eqref{eq:matrix-rd>1}, \eqref{eq:matrix-rd=1} in the singular
case.

Then the function $u^\eps$ satisfies
\begin{equation*}
  \left\{\begin{array}{ll}
      u^\eps \to \bar{m}_+  & \text{ in } \{ u > 0 \}\\
      u^\eps \to \bar{m}_- & \text{ in } \{ u < 0 \} 
    \end{array}\right.  \quad \text{ as } \quad \eps \to 0
\end{equation*}
where $\bar{m}_\pm$ denote the stable zeros of $f$; 
moreover both limits are local uniform.
\end{thm}
\begin{proof}
In order to prove this theorem, we use the abstract method developed in \cite{bs98}
and \cite{bd03}. Consider  two open sets defined as
\begin{eqnarray*}
\Omega^1 & = & \mathrm{Int} \{ (t,x) \in (0,+\infty) \times \R^N : \liminf{}_* 
\frac{u^\eps-\bar{m}_+}\eta \ge 0 \} 
\subset (0,+\infty) \times \R^N \\
\Omega^2 & = & \mathrm{Int} \{ (t,x) \in (0,+\infty) \times \R^N : \limsup{}_* 
\frac{u^\eps - \bar{m}_-}\eta \le 0  \} 
\subset (0,+\infty) \times \R^N 
\end{eqnarray*}
where interior is considered with respect to $(0,+\infty) \times \R^N$.

We next define their traces at initial time by considering the lower
semi-continuous function $\chi = \un_{\Omega^1} - \un_{(\Omega^1)^c}$
and the upper semi-continuous function $\bar{\chi} =
\un_{(\Omega^2)^c} - \un_{\Omega^2}$ defined on $(0,+\infty) \times
\R^N$.  They can be extended at $t=0$ by setting $\chi(0,x)=
\liminf_{t\to 0,y\to x} \chi (t,y)$ and $\bar{\chi}(0,x)=
\limsup_{t\to 0,y\to x} \bar{\chi} (t,y)$. We now define
$$
\Omega^1_0 = \{ x \in \R^N : \chi (0,x) =1 \} \quad \text{ and } \quad
\Omega^2_0 = \{ x \in \R^N : \bar{\chi} (0,x) = -1 \} \; .
$$
The method developed in \cite{bs98} consists in proving the following propositions. 
\begin{prop}[Initial time] \label{prop:omega01}
The set  $\{ x \in \R^N : d_0 (x) > 0 \}$ is contained in $\Omega^1_0$. 
Similarly, the set $\{ x \in \R^N : d_0 (x) < 0 \}$ is contained in $\Omega^2_0$. 
\end{prop}
\begin{prop}[Propagation] \label{prop:gen-flow}
The set $\Omega^1$ (resp. $\overline{\Omega^2}$) defined above
 is a generalized super-flow (resp. sub-flow) of \eqref{eq:amcm}.
\end{prop}
The proofs of both propositions are postponed.
Applying next \cite[Corollary~2.1]{bd03} (see also \cite[Corollary~3.1]{bs98}),
we conclude the proof of Theorem~\ref{thm:rd>1-convergence}.
\end{proof}
It remains to prove Propositions~\ref{prop:omega01} and
\ref{prop:gen-flow}.  Both rely on the construction of barriers for
smooth fronts; in our case, the term ``barrier'' refers to a sub- or
super-solution of the fractional diffusion-reaction equation. 

\subsection{Proofs of Propositions~\ref{prop:omega01} and \ref{prop:gen-flow}}

As we shall see it, proofs of Propositions~\ref{prop:omega01} and
\ref{prop:gen-flow} reduce to the proof of the following one.
\begin{prop}[Construction of a barrier]\label{prop:barrier}
  Given $t_0>0$ and $x_0 \in \R^N$, consider a smooth function $\phi :
  (0,+\infty) \times \R^N \to \R$ such that \ref{0}, \ref{i}, \ref{ii}
  from Definition~\ref{def:gen-flow} are satisfied.  For all $\beta
  >0$, there exists a sub-solution $U^{\eps,\beta}$ of \eqref{eq:rd>1} if
  $\alpha>1$ and \eqref{eq:rd=1} if $\alpha =1$ such that
\begin{equation}\label{eq:clU}
U^{\eps,\beta} (t_0, x) \le (m_+(-\beta \eta)- \beta \eta) \un_{\{ d(t_0,\cdot) \ge \beta \}}
+ m_-(-\beta \eta) \un_{\{  d(t_0,\cdot ) < \beta \}}  \text{ for } x \in \R^N 
\end{equation}
where $d(t,x)$ denotes the signed distance to the set $\{y :\phi (s,y)=0\}$ which has the same
signs as $\phi$ and $m_\pm(-\beta \eta)$ are the stable equilibria of $f + \beta\eta$. Moreover,
if $d(t,x) > -2\beta$, then
\begin{equation}\label{eq:limfracU}
\liminf{}_* \frac{U^{\eps,\beta} -\bar{m}_+}\eta (t,x)\ge -(C_f +2)\beta 
\end{equation}
where $C_f$ appears in \eqref{eq:m-eps-bar}. 
\end{prop}

We first derive Proposition~\ref{prop:gen-flow} from the construction of the barrier.
\begin{proof}[Proof of Proposition~\ref{prop:gen-flow}]
  Let us explain why Proposition~\ref{prop:barrier} together with the
  comparison principle for \eqref{eq:rd} yield the desired result.  

We consider a smooth function $\phi$ such that \ref{0}-\ref{iii} hold
true. Let $d$ denote the signed distance function to $\{ \phi = 0 \}$. 
We  derive from \ref{iii} that 
$$
\{ d(t_0,\cdot) \ge 0 \} \subset \{ \phi (t_0,\cdot) \ge 0 \} \subset \Omega^1_{t_0}
= \{ \liminf{}_* \frac{u^\eps-\bar{m}_+}\eta \ge 0 \} \, .
$$
For the sake of clarity, $m_\pm$ denotes $m_\pm(-\beta \eta)$. 
And we will do so in the remaining of the paper. 
Hence
$$
\{ d(t_0,\cdot) \ge \beta \} \subset \mathrm{int} \{ \liminf{}_* \frac{u^\eps-\bar{m}_+}\eta \ge 0  \} \, .
$$
Hence by using \ref{0}, we conclude that $u^\eps$ satisfies on one hand
$$
u^\eps \ge \bar{m}_+-\beta \eta \ge m_+-\beta \eta \quad \text{ in } \{ d(t_0,\cdot) \ge \beta \} \, .
$$
On the other hand, since $\bar{m}_-$ is a trivial solution of the diffusion-reaction equation,
we have $u^\eps \ge \bar{m}_- \ge m_-$. We thus conclude that $u^\eps$ satisfies
\begin{equation}\label{eq:clphieps}
u^\eps (t_0, x) \ge (m_+- \beta \eta) \un_{\{ d (t_0,\cdot) \ge \beta \}}
+ m_- \un_{\{ d(t_0,\cdot) < \beta \}}  \text{ for } x \in \R^N \, , 
\end{equation}
We now use Proposition~\ref{prop:barrier} in order to get a
sub-solution $U^\eps$ with the desired properties.  Combining
\eqref{eq:clU} and \eqref{eq:clphieps} yields that $U^\eps \le u^\eps$
at $t=t_0$.  We thus conclude by using the comparison principle for
\eqref{eq:rd} that $U^\eps \le u^\eps$ on $[t_0,t_0+h] \times \R^N$
and \eqref{eq:limfracU} implies that
$$
\liminf{}_*  \frac{u^\eps (t_0+h,x)-\bar{m}_+}\eta \ge -(C_f+2) \beta
$$
as soon as  $d(t_0+h,x) > 2\beta$. Since $\beta$ is arbitrary, the proof is complete. 
\end{proof}
We now prove Proposition~\ref{prop:omega01}.
\begin{proof}[Proof of Proposition~\ref{prop:omega01}]
We only prove the result for $\Omega^1_0$ since the proof for $\Omega^2_0$ is similar. 

Let $x_0$ be such that $d_0 (x_0)=:2 \delta > 0$. We have to prove that $x_0 \in \Omega^1_0$.
In other words, for all $(t,x)$ in a neighbourhood  of $(0,x_0)$, we would like to prove
$$
\liminf{}_* \frac{u^\eps - \bar{m}_+}\eta (t,x) \ge  0 \,.
$$
In order to get such a result, we construct for any small $\beta >0$ 
a subsolution $U^{\eps,\beta}$ of \eqref{eq:rd} such that
$$
U^{\eps,\beta} (0,x) \le u^\eps (0,x)  
$$ 
and satisfying \eqref{eq:limfracU} for some function $d(t,x)$ such that
$\{ d > 2 \beta \}$ contains a neighbourhood of $(0,x_0)$. 

There exists $r>0$ such that for any $x \in \bar{B} (x_0,r)$, $d_0 (x)
\ge \delta > 0$.  Consider next the smooth function
$$
\phi (t,x) = (r-Ct)_+^2 - |x-x_0|^2 \, .
$$
The associated distance function is given by the following formula
$$
d (t,x) = r - C t - |x-x_0| \, .
$$ 
Remark that $\{ d > 2 \beta \} = \cup_{t\ge 0} \{t \} \times B(x_0,r-Ct-2\beta)$.

We claim that  \eqref{eq:clphieps} holds true with $t_0=0$.
Indeed, when $d (0,x) \ge \beta$, we know that $d_0 (x) \ge \delta$ and
this  implies that $u^\eps (0,x) \ge m_+ - \beta \eta$ for $\eps$ small enough
as showed now
\begin{eqnarray*}
u^\eps (0,x) &\ge& q^0\left( \frac{d_0(x)}\eps, Dd_0 (x) \right) \\
& \ge &q^0 \left( \frac{\delta}\eps, Dd_0 (x) \right) \\
& \ge &\bar{m}_+ + o (\eps^{1+\alpha})  \\
& \ge &m_+ + o (\eps^{1+\alpha})  \\
& \ge &m_+ - \beta \eta \, .
\end{eqnarray*} 

Notice that \ref{0} and \ref{ii} are satisfied. As far as \ref{i} is
concerned, it is only used in the construction of the barrier in order
to get \eqref{eq:d} below. We thus have to prove that we can choose $C
>0$ such that \eqref{eq:d} also holds true.  The constant $C$ is
chosen as follows
$$
C \ge \sup_{e \in \S} \left( - \mu (e) \mathrm{Tr} (A(e) ) - \bar{c} (e) \right) + \frac{\delta_\phi}2 
$$
and \eqref{eq:d} holds true for $\gamma$ and $h$ small enough. 
\end{proof}

\section{Construction of the barrier}

\label{sec:construction}

This section is devoted to the proof of Proposition~\ref{prop:barrier}. 
\begin{proof}[Proof of Proposition~\ref{prop:barrier}]
The proof proceeds in several steps. We first construct a sub-solution $U$ of
the diffusion-reaction equation close to the smooth front and we then extend
it to the whole space. 

\paragraph{A barrier close to the front.}
Using \ref{0}, \ref{i} and \ref{ii}, we know that
there exists $\gamma >0$ such that $d$ is smooth on the set 
$$
Q_\gamma = \{ |d| < \gamma \} 
$$
and $D\phi (s,x) \neq 0$ on $Q_\gamma$ and 
\begin{equation}\label{eq:d}
\partial_t d \le \mu (Dd) \mathrm{Tr} (A(Dd) D^2d) - \frac{\delta_\phi}{2}  \quad \text{ in } Q_\gamma \, .
\end{equation}
We used the fact that $|Dd|=1$ in $Q_\gamma$ which also implies that
$D^2 d Dd =0$ in $Q_\gamma$.  For $\beta \le \gamma /2$, we next
define a ``barrier'' as
\begin{multline*}
U (t,x) = q \left( \frac{d(t,x)-2 \beta}\eps , \D (t,x),-\beta h \right) 
\\+ h Q \left( \frac{d(t,x)-2\beta}\eps ,\D(t,x), t,x,-\beta h \right) - 2 \beta h
\end{multline*}
where $h >0$ will be chosen later, $q$ denotes the traveling wave
given by Assumption~\ref{assum:atw}; the function $\D$ is assumed to
be smooth, to coincide with $Dd$ in $Q_\gamma$ and to be such that
$\frac12\le |\D|\le \frac32$.  Let us point out that we would like to
choose $\D=Dd$ but this function is not well defined everywhere away
from the front and even if we  prove that $U$ is a subsolution
close to the front, $U$ has to be defined everywhere since the
diffusion-reaction equation is not local. As far as the function $Q$
is concerned, it will be chosen later.

\paragraph{Plugging the barrier into the diffusion-reaction equation.}
In order to prove that the barrier we introduced in the previous
step is a sub-solution of the diffusion-reaction equation close to
the front, we first plug it into the equation. 
\begin{lem}
If $h = \eta$ and $\beta \le \bar \beta$ (depending only on $\delta_\phi$), then
the function $U$ satisfies the following inequality in $Q_\gamma$
\begin{multline}
\label{eq:U2}
\partial_t U + \frac{1}{\eps \eta}\{ - \mathcal{I}^\eps_\alpha U +  f (U) \} 
\le \frac{\dot{q}(r,e)}\eps \bigg[ \mu (e) \mathrm{Tr} (A(e) B) -\mu_\eps \bar{a}_\eps 
   - \frac{\delta_\phi}4 \bigg] - \frac\beta\eps  \\ 
+ \frac1\eps \left[ \mathcal{L} Q (r,e,t,x)- (a_\eps (r,e,t,x)- \dot{q} (r,e) \mu_\eps \bar{a}_\eps ) \right] 
+ (\mathrm{err})
\end{multline}
where $r = \frac{d(t,x)-2\beta}\eps$, $e = Dd (t,x)$, $B=D^2 d(t,x)$, $\mu_\eps,\bar{a}_\eps \in \R$ are 
two real numbers to be chosen later and
\begin{equation}\label{aeps}
  a_\eps (r,e,t,x) = \frac1\eta \int \bigg\{  q ( r + \frac{d(t,x+\eps z) - d(t,x)}\eps, e)
  -  q ( r + e \cdot z, e)  \bigg\} J(z) dz \, .
\end{equation}
As far as error terms are concerned, we have
\begin{eqnarray}
  (\mathrm{err}) =  \frac1{\eps \eta} R [T_q] + \frac1\eps R[T_Q] 
  + \frac1\eps \left(- \dot{q} \frac{\delta_\phi}8  
    - 2 f'(q) \beta  \right)  + o(\eps^{-1}) 
\label{eq:error-terms}
\end{eqnarray}
and 
\begin{multline} \label{rtq}
R[T_q]  = \int \bigg\{  q \left( \frac{d (t, x+ \eps z)-2\beta}\eps, e \right) \\
- q \left( \frac{d (t, x+ \eps z)-2\beta}\eps, \D (t,x+\eps z) \right) 
+  B D_e q (r,e) \cdot \eps z \un_B (\eps z) \bigg\} J(z) dz, 
\end{multline}
\begin{multline} \label{rtQ}
R[T_Q] =  \int \bigg\{ Q(r+ e \cdot z, e,t,x) \\
- Q \left( \frac{d(t,x+\eps z)-2\beta}\eps ,
\D(t,x+\eps z) ,t,x+\eps z \right)  \\
 +  (B D_e Q + D_x Q) \cdot \eps z \un_B(\eps z) \bigg\} J(z) dz \, .
\end{multline}
\end{lem}
\begin{proof}
We compute the quantity $\partial_t U + (\eps \eta)^{-1} \{ - \mathcal{I}^\eps_\alpha U  + f(U) \}$.

The fact that $q$ is a traveling wave (see \eqref{eq:tw-rd}) together with a
uniform bound on $Q$ with respect to all its variables (we will choose
$Q$ below so that it satisfies such a condition) permits to get
\begin{eqnarray*}
f(U) &=& f(q) + h f'(q) Q - 2 f'(q) \beta h  + O (h^2) \\
& = & -\beta h + \mathcal{I}_e [q] - c (e,-\beta h) \dot{q} +  h f'(q) Q 
- 2 f'(q) \beta h + O (h^2)  \, .
\end{eqnarray*}
Rearranging terms, we thus obtain, for $(t,x) \in Q_\gamma$,
\begin{multline} \label{eq:U} \frac{1}{\eps \eta}\{ -
  \mathcal{I}^\eps_\alpha U + f (U) \} \le \frac{\dot{q}}\eps \left[ -
    \frac{c(e,-\beta h)}\eta (0,D d) \right] - \frac{\beta h}{\eps
    \eta} - 2 f'(q) \beta \frac{h}{\eps \eta} \\ 
+ \frac1{\eps \eta}(\mathcal{I}_e[q] +hf'(q) Q- \mathcal{I}^\eps_\alpha U)
+ O (\frac{h^2}{\eps \eta})
\end{multline}
We immediately see from this computation and in view of \eqref{eq:cbar}
that $h$ must be chosen as follows
$$
h = \eta \, .
$$

We next write 
\begin{equation}\label{eq:integral terms}
\mathcal{I}_e[q](r,e) +\eta f'(q)Q- \mathcal{I}^\eps_\alpha U (t,x)= \eta T_Q+T_q 
\end{equation}
where
\begin{equation}
\label{eq:B}
 T_Q =  f'\left( q \right) Q  - \mathcal{I}^\eps_\alpha ( Q(\eps^{-1}(d-2\beta),Dd,t,\cdot))
= \mathcal{L} Q -c \dot{Q} + R [T_Q] 
\end{equation}
and
\begin{eqnarray*}
 T_q &=& \mathcal{I}_e [q](r) - \mathcal{I}^\eps_\alpha [q(\eps^{-1} (d-2\beta), \D)](x) \\
& =&   \int \bigg\{  q ( r + e
  \cdot z, e) - q  (r, e) - \dot{q} (r,e) e \cdot z \un_B (z)\bigg\} J(z) dz  \\
& &  -  \int \bigg\{  
 q \left( \frac{d (t, x+ \eps z)-2\beta}\eps, \D (t,x+\eps z) \right)
- q (r,e) \\ && \qquad \quad - \left( \dot{q} (r,e) \frac{e}\eps + B D_e
  q (r,e) \right) \cdot \eps z \un_B (\eps z) \bigg\} J(z) dz \, .
\end{eqnarray*}
Hence, $T_q$ can be written as follows
\begin{equation}\label{eq:A}
T_q =   - \eta a_\eps + R[T_q] \, .
\end{equation}

We now compute the time derivative of the barrier $U$. We use \eqref{eq:d} in order to get
\begin{multline}\label{eq:timederU}
\partial_t U =  \frac{\dot{q}}\eps \left[ \partial_t d \right] 
+ (D_e q + \eta D_e Q) \cdot D (\partial_t d) +  \eta \eps^{-1} \dot{Q} \partial_t d + \eta
\partial_t Q \\
\le \frac{\dot{q}}\eps \left[ \mu (e) \mathrm{Tr} (A(e) B) - \frac{\delta_\phi}2 \right] 
+ (D_e q + \eta D_e Q) \cdot D (\partial_t d) +  \eta \eps^{-1} \dot{Q} \partial_t d + \eta
\partial_t Q .
\end{multline}

We next combine  \eqref{eq:U}, \eqref{eq:integral
  terms}, \eqref{eq:B}, \eqref{eq:A} and \eqref{eq:timederU} to get
\begin{multline*}
  \partial_t U + \frac{1}{\eps \eta}\{ - \mathcal{I}^\eps_\alpha U + f
  (U) \} \le \frac{\dot{q}}\eps \left[ \mu (e) \mathrm{Tr} (A(e) B) - \frac{\delta_\phi}4
-   \frac{c(e,-\beta \eta)}\eta - \frac{\delta_\phi}8 \right]-
  \frac\beta\eps \\ + \frac1\eps \left[ - a_\eps + \mathcal{L} Q
  \right] + (\mathrm{err})
\end{multline*}
with 
\begin{multline*}
(\mathrm{err})=
\frac1{\eps \eta} R [T_q] + \frac1\eps R[T_Q] - \frac{c}\eta \dot{Q} 
+ \frac1\eps \left(- \dot{q} \frac{\delta_\phi}8  - 2 f'(q) \beta  \right) \\
 + (D_e q + \eta D_e Q) \cdot D (\partial_t 
d) +  \eta \eps^{-1} \dot{Q} \partial_t d + \eta \partial_t Q + o(\eps^{-1}) \, .
\end{multline*}
Using \eqref{eq:cbar}, \eqref{eq:q-estim} and \eqref{eq:Q-estim}, we finally get \eqref{eq:U2} with the
associated error term. 
\end{proof}

\paragraph{Estimating  error terms.}
In this paragraph, we prove that the right hand side of \eqref{eq:U2} is non-positive. 
We first construct a corrector $Q$ in order to handle oscillating terms.
\begin{lem}[Choice of the corrector $Q$]\label{lem:corrector}
There exist $\mu_\eps,\bar{a}_\eps \in \R$ such that there exists $Q$ satisfying 
\begin{equation} \label{cellQ}
\mathcal{L} Q  = a_\eps  - \dot{q} \mu_\eps  \bar{a}_\eps  \; . 
\end{equation}
\end{lem}
\begin{proof}
  In view of Proposition~\ref{assum:Q}, it is enough to choose
  $\mu_\eps$ and $\bar{a}_\eps$ such that
$$
\int (a_\eps (\xi) - \mu_\eps \bar{a}_\eps \dot{q} (\xi) ) \dot{q} (\xi) d\xi =0. 
$$
The following choices permit to ensure such a condition
$$
\mu_\eps  (e) = \left( \int \dot{q}^2 (\xi,e)d\xi \right)^{-1} 
$$
and
\begin{multline*}
\bar{a}_\eps (e,t,x)  =  \int \dot{q} (\xi,e) a_\eps (\xi,e,t,x) d\xi \\
= \frac1\eta \iint \dot{q} (\xi,e) \bigg\{  
q ( \xi + e \cdot z + \eps W (t,x,z), e) 
-  q ( \xi + e \cdot z, e)  \bigg\} J(z) d\xi dz 
\end{multline*}
with 
$$
W(t,x,z) = \frac{1}{\eps^2} [ d(t,x+\eps z) -d(t,x) - \eps Dd (t,x) \cdot z] \, .
$$
\end{proof}
\begin{rem} \label{rem:choice-h}
The choice of $h$ when rescaling fractional diffusion-reaction equations \eqref{eq:rd}
is made such that $\bar{a}_\eps$ has a limit as $\eps \to 0$. 
\end{rem}
The following lemma
is the core of the proof of Theorem~\ref{thm:rd>1-convergence} and its proof
is rather involved. This is the reason why we postpone it until Section~\ref{sec:lem}. 
\begin{lem}[Uniform convergence of approximate coefficients (I)] 
\label{lem:a-eps-conv}
As $\eps \to 0$,
$$
\bar{a}_\eps (e,t,x)   \to  \mathrm{tr} \; (A (e) D^2 d(t,x)) 
$$
and the limit is uniform with respect to $(e,t,x) \in \S \times  Q_\gamma$. 
\end{lem}
We next treat error terms appearing in $(\mathrm{err})$.  
\begin{lem}[Error terms $(\mathrm{err})$] \label{lem:estim-errors}
We have 
\begin{equation} \label{eq:estim-rtq}
 R[T_q] = o(\eps^\alpha) = o (\eta) 
\end{equation}
and
\begin{equation} \label{eq:estim-rtQ}
R[T_Q] =o(1)
\end{equation}
uniformly in $(e,t,x) \in \S \times Q_\gamma$ and
 for all $r \in \R$ and $\beta \le \bar \beta = \bar \beta ( \delta_\phi)$
\begin{equation}\label{estim:additional-error-term}
- \dot{q} (r) \frac{\delta_\phi}8  - 2 f'(q(r)) \beta  \le 0 \, .
\end{equation}
\end{lem}
\begin{proof}
We first prove \eqref{eq:estim-rtq}. Through a change of variables, we get
\begin{multline*}
R[T_q] = \eps^{\alpha} \int \bigg\{ q \left( \frac{d (t,x+ \bar{z})-2\beta}\eps, \mathcal{D} (t,x)
\right) \\- q \left( \frac{d (t,x +\bar{z})-2\beta}\eps, \mathcal{D} (t,x+\bar{z}) \right) 
- B D_e q (r,e) \cdot \bar{z} \un_B (\bar{z})
\bigg\} J_\eps (\bar{z}) d \bar{z}
\end{multline*}
where $J_\eps (\bar{z}) = \eps^{-(N+\alpha)} J (\eps^{-1} \bar{z})$. By using \eqref{cond1:J} and
\eqref{eq:q-estim}, dominated convergence theorem permits to conclude. 

We next turn to the proof of \eqref{eq:estim-rtQ}. To prove it, we first write
\begin{equation}\label{cut-rtQ}
R[T_Q]  =  R^1[T_Q] - R^2 [T_Q]
\end{equation}
with 
\begin{eqnarray*}
R^1[T_Q] &= & 
\int \bigg\{ Q(r+ e \cdot z, e,t,x) - Q \left( \frac{d(t,x+\eps z)-2\beta}\eps ,
e ,t,x \right) \bigg\} J(z) dz \; , \\
R^2[T_Q] &= &  \int \bigg\{ Q \left(r+e\cdot z ,
\D (t,x+\eps z) ,t,x+\eps z \right)  \\
&&-Q(r+ e \cdot z, e,t,x) 
- (B D_e Q + D_x Q) \cdot \eps z \un_B (\eps z)
\bigg\} J(z) dz \; .
\end{eqnarray*}

As far as $R^1[T_Q]$ is concerned, we can write for any $R>0$  
\begin{eqnarray*}
R^1 [T_Q] &\le &\|\dot{Q} \|_\infty \frac12 \|D^2 d\|_{L^\infty(B(x,\eps R))} \eps
\int_{|z|\le R } |z|^2 J(z) dz \\ && + 2 \| Q\|_\infty \int_{|z| \ge R} J(z) dz \\
& \le & C \|\dot{Q} \|_\infty \frac12 \|D^2 d\|_{L^\infty(B(x,\eps R))} \eps R^{2-\alpha} 
+  2 C \| Q\|_\infty R^{-\alpha} 
\end{eqnarray*}
where we used \eqref{cond1:J} to get the second inequality. 
Choose now $R$ such that $\eps R \le 1$, $R \to +\infty$ and $\eps R^{2-\alpha} \to 0$;
for instance $R = \eps^{-1/2}$ permits to conclude in this case. 

As far as $R^2 [T_Q]$ is concerned, we use once again \eqref{cond1:J} in order
to write
\begin{multline*}
|R^2 [T_Q]| \le \eps^\alpha \int \bigg| Q \left( \dots ,
\D (t,x+\bar z) ,t,x+\bar z \right)  \\
-Q(\dots, \D (t,x),t,x) 
- (B D_e Q + D_x Q) \cdot \bar z \un_B (\bar z)
\bigg| C_J \frac{d \bar z}{|\bar z|^{N+\alpha}} \\
\le C \eps^\alpha 
\end{multline*}
where $C$ depends on $\sup_h \{ \|Q\|_\infty + \| D^2_{e,e} Q \|_\infty + \| D^2_{x,x} Q \|_\infty \}$
that is bounded by assumption (see Estimate~\eqref{eq:Q-estim}).

It remains to prove \eqref{estim:additional-error-term}. 
It is enough to prove that there exists a constant $C_{tw}$ which does not depend on $h$
and such that for all $r \in \R$
\begin{equation}\label{eq:q-newone}
\dot{q} (r) + C_{tw} f'(q(r)) \ge 0 \, .
\end{equation}
This inequality is trivial when $f'(q(r)) \ge 0$. Hence, we consider $r$ such that 
$f'(q(r)) \le 0$, that is to say
$$
\bar{m}_- < \bar{q}_- \le q(r) \le \bar{q}_+ < \bar{m}_+
$$
for some constants $\bar{q}_\pm$ which do not depend on $h$. If $r$ satisfies the previous
inequality, we deduce from \eqref{eq:q-decay} that 
$$
|r| \le R
$$
for some constant $R$ which does not depend either on $h$. 
Now \eqref{eq:q-newone} is clear. It is enough to find an estimate from below
for $\dot{q}$ on $[-R,R]$ which does not depend on $h$. The proof of the lemma 
is now complete. 
\end{proof}
Using Lemmata~\ref{lem:corrector}, \ref{lem:a-eps-conv} and \ref{lem:estim-errors}  
we derive from \eqref{eq:U2} the following inequality
\begin{equation}\label{eq:U5}
\partial_t U + \frac{1}{\eps \eta}\{ - \mathcal{I}^\eps_\alpha U +  f (U) \} 
\le  \frac{-\beta}{\eps} + o(\frac1\eps) \le \frac{-\beta}{2\eps} \; . 
\end{equation} 
\medskip

\paragraph{Extension of the barrier away from the front.}
The remaining of the construction of the barrier consists in extending
the subsolution $U$ we constructed before in order that it is a
subsolution in $[t_0,t_0+h] \times \bar{B}(x_0,r)$ (in particular, far
from the front).  More precisely, we modify $U$ far from the
front. Following \cite{bs98,bd03}, we proceed in two steps. We first
extend it on $\{ d \le \gamma \}$ by $m_-$ and then
extend it on $ \{ d \ge \gamma \}$ by $m_+ - \beta \eta$
The difficulty is to keep it a subsolution. We do this by truncating
properly $U$. Truncating it from below by $m_-$ is easy
but truncating it from above by $m_+ -\beta \eta$ is more
delicate.  \medskip

\noindent \textsc{Upper estimates for $U$.}
We start by estimating from above the ``barrier'' function $U$ we constructed before.
We claim that the following inequalities hold true
\begin{eqnarray}\label{prop:u1}
 U (t,x)   &\le& m_- \quad \text{ in } \{ d \le \beta \} \, , \\
\label{prop:u2}
U (t,x) &\le& m_+  - \beta \eta \, .
\end{eqnarray}
We first justify \eqref{prop:u1}. In view of the definition of $U$, we use 
 \eqref{eq:q-decay} and \eqref{eq:Q-decay},  in 
order to get
$$
U (t,x) \le m_- + O(\eps^{1+\alpha}) - 2 \beta \eta  \le m_-\, .
$$

We next justify \eqref{prop:u2} by adapting an argument from \cite{bs98}. 
First, \eqref{eq:Q-decay} implies that there exists $\bar{c} >0$ which does not depend on 
$h$ such that we have for  $|r| \ge \bar{c}$
$$
| Q (r,e,t,x) | \le \beta \, .
$$
Next, we claim that there exists $\nu (\bar{c}) >0$ such that 
we have for $|r| \le \bar{c}$
\begin{equation}\label{eq:q-prop}
q (r) \le m_+  - \nu (\bar{c}) \, .
\end{equation}
Now if $|d(t,x)-2\beta| \ge \eps \bar{c}$, then
$$
U (t,x) \le m_+ + \beta \eta - 2 \beta \eta \le m_+ - \beta \eta\, .
$$
In the other case, $|d(t,x)-2\beta| \le \eps \bar{c}$, then
$$
U (t,x) \le m_+ - \nu (\bar{c}) + \eta \|Q \|_\infty - 2 \beta \eta \le m_+ - \beta \eta
$$
as soon as $\eta \|Q\|_\infty \le \nu (\bar{c})$. 
\medskip

\noindent \textsc{Definition of $\bar{U}$.}
We define for $(t,x) \in [t_0,t_0+h] \times \R^N$
$$
\bar{U} (t,x) = \max ( U (t,x), m_- ) \; .
$$
From \eqref{prop:u1} and \eqref{prop:u2}, we get
\begin{eqnarray}
\label{prop:baru1}
\bar U &=& m_- \quad \text{ in } \{ d \le \beta\} \, , \\
\label{prop:baru2}
m_- &\le& \bar U \le m_+ - \beta \eta \, .
\end{eqnarray}
On one hand, a classical argument implies that $\bar{U}$ is a
subsolution of \eqref{eq:rd} on $Q_\gamma= \{ -\gamma \le d \le \gamma\}$ since it is the maximum of
two subsolutions.  On the other hand, \eqref{prop:u1} implies that $\bar{U} (t,x) =
m_-$ on $\{ d  \le  -\gamma /3 \}$.  Thus $\bar{U}$ is a subsolution on $\{d \le \gamma \}$.

We also shed light on the fact that $\bar{U}$ satisfies \eqref{eq:U5} at points of $Q_\gamma$ 
where $\bar{U} = U$.  
This will be used later on. We will also need the  following lemma.
\begin{lem}[Gradient estimate for the barrier]\label{lem:estim-nablaubar}
There exists $\bar{C} >0$ such that for all $(t,x) \in [t_0,t_0+h] \times \R^N$ and all $\eps >$ small enough,
$$
| \eps D \bar{U} (t,x) | \le \bar{C} \; .
$$
\end{lem}
\begin{proof}
  This is a simple consequence of the definition of $\bar{U}$ and of
  Estimates~\eqref{eq:q-estim} and \eqref{eq:Q-estim}.
\end{proof}

\noindent \textsc{Definition of $U^{\eps,\beta}$.}
We finally define 
$$
U^{\eps,\beta} = \left\{ \begin{array}{ll}
    \psi(d) \bar{U} + (1-\psi(d)) (m_+-\beta \eta)  & \text{ if } d < \gamma \; ,\\
    m_+ -\beta \eta & \text{ if } d \ge \gamma
\end{array} \right. 
$$
where $\psi:\R \to [0,1]$ is a smooth function such that $\psi(r)=1$
if $r \le \gamma/2$, $\psi(r) = 0$ if $r \ge 3\gamma/4$. We will see
below that it is convenient to assume additionally that $\psi
(5\gamma/8+r)= 1-\psi (5\gamma/8-r)$.
We deduce from properties of $\psi$ and \eqref{prop:baru1} and \eqref{prop:baru2} that
\begin{align}
\label{prop:ueps1}
U^{\eps,\beta} &= m_-  & \text{ in } \{d \le \beta\} \, \\
\label{prop:ueps2}
U^{\eps,\beta} &= m_+ -\beta \eta & \text{ in } \{d \ge 3\gamma/4 \} \, \\
\label{prop:ueps3}
m_- \le U^{\eps,\beta} &\le m_+ -\beta \eta \, . &
\end{align}
In particular, \eqref{eq:clU} is clearly satisfied. 

We also deduce from \eqref{prop:ueps3} and the definitions of $\bar U$ and $U^{\eps,\beta}$ that
\begin{equation}\label{eq:ueps-ubar}
U \le \bar{U} \le U^{\eps,\beta} \, .
\end{equation}
In particular, if $d(t,x) > 2 \beta$, then $d>2\beta$ in a
neighbourhood $\mathcal{V}$ of $(t,x)$. In particular, 
$$
U^{\eps,\beta} \ge U  = m_+ + o(\eta) - 2 \beta \eta \quad \text{ in } \mathcal{V} \, .
$$
Using now \eqref{eq:m-eps-bar}, we deduce that \eqref{eq:limfracU} also holds true.
\medskip

\paragraph{The barrier $U^{\eps,\beta}$ is a subsolution of
  \eqref{eq:rd} on $[t_0, t_0+h] \times \R^N$.}  We distinguish three cases.

 Consider first a
point $(t,x)$ such that $d(t,x) < \gamma /2$. In this case,
$U^{\eps,\beta} (s,y) = \bar{U}(s,y)$ in a neighbourhood of $(t,x)$
and this implies $\partial_t U^{\eps,\beta} (t,x) = \partial_t \bar{U}
(t,x)$ (in the viscosity sense).  In order to prove that
$U^{\eps,\beta}$ is a subsolution of \eqref{eq:rd} at $(t,x)$ it is
enough to prove that $\mathcal{I}^\eps_\alpha \bar{U}(t,x) \le
\mathcal{I}^\eps_\alpha U^{\eps,\beta}(t,x)$ since $\bar{U}$ is a
subsolution.  Such an inequality is a consequence of
\eqref{eq:ueps-ubar} and the fact that $U^{\eps,\beta} (t,x) =
\bar{U}(t,x)$.

Consider next a point $(t,x)$ such that $d(t,x) > 3 \gamma /4$. In
this case, there exists $r_0 >0$ such that $d(s,y) > 3 \gamma /4$ for
$y \in B((t,x),r_0)$. Consequently, $U^{\eps,\beta} (s,y) = m_+
- \beta \eta$ for $(s,y) \in B((t,x),r_0)$.  This yields that
$\partial_t U^{\eps,\beta} (t,x) =0$ (in the viscosity sense) so we
have to prove that
$$
f(U^{\eps,\beta} (t,x)) \le \mathcal{I}^\eps_\alpha (U^{\eps,\beta}) (t,x) \, .
$$
To get the previous inequality, on one hand, we have
\begin{eqnarray*}
f(U^{\eps,\beta} (t,x)) &=& f(m_+ -\beta \eta) \\
&= & f(m_+ (-\beta \eta)) - f'(m_+ - \theta \beta \eta) \beta \eta\\
&\le&  -   \beta \eta 
\end{eqnarray*}
 and on the other hand,
\begin{eqnarray*}
\mathcal{I}^\eps_\alpha (U^{\eps,\beta}) (t,x) &=&
 \int [ U^{\eps,\beta} (t,x+\eps z) - (m_+ -\beta \eta)] J(z) dz \\
&=& \int_{|z| \ge r_0 / \eps } [U^{\eps,\beta} (t,x+\eps z) - (m_+-\beta \eta)] J(z) dz \\
& \ge & - C \int_{|z| \ge r_0 / \eps } J(z) dz = - C \eps^\alpha  \ge - \beta \eta 
\end{eqnarray*}
in view of the definition of $\eta$. Notice that this argument fails in the case $\alpha <1$. 

Finally, we consider $(t,x)$ such that $ \gamma/2 \le d(t,x) \le 3
\gamma /4$.  

It is convenient to introduce $\psi_d (x)= \psi ( d (x))$.  Remark
that $\bar U =U$ in a neighbourhood of $(t,x)$. Hence, we mentioned
above that $\bar U$ satisfies \eqref{eq:U5} at $(t,x)$ 
\begin{equation}\label{eq:barU}
  \partial_t \bar U + \frac{1}{\eps \eta}\{ - \mathcal{I}^\eps_\alpha \bar U +  f (\bar U) \} 
  \le \frac{-\beta}{2\eps} \, .
\end{equation}
We use \eqref{eq:barU} and compute (in the viscosity sense)
\begin{multline}
  \partial_t U^{\eps,\beta} + \frac1{\eps \eta}
  \left(f(U^{\eps,\beta}) - \mathcal{I}^\eps_\alpha U^{\eps,\beta}
  \right) \\ = \psi' \times (\partial_t d) \times (\bar{U} -
  m_++\beta \eta) + \psi_d \partial_t \bar{U} + \frac1{\eps
    \eta} f(U^{\eps,\beta}) - \frac1{\eps \eta} I^\eps_\alpha
  (U^{\eps,\beta} ) \\ \label{eq:presque} \le C(\psi) \bigg|\bar U - (m^+ (-\beta \eta) -\beta \eta) \bigg| +
  \frac1{\eps \eta} \bigg[ f(U^{\eps,\beta})- \psi_d f(\bar{U}) \bigg]  \\
+ \frac1{\eps \eta} \bigg[\psi_d
  \mathcal{I}^\eps_\alpha \bar{U}- \mathcal{I}^\eps_\alpha
  U^{\eps,\beta} \bigg ]- \psi_d  \frac\beta{2\eps}
\end{multline}
where $C(\psi)$ only depends on $\psi$ and $\gamma$.  We now estimate
each term of the right hand side of \eqref{eq:presque}.

First, we derive directly from the equality $\bar U = U$ and the very definition 
of $U$ the following lemma
\begin{lem}\label{lem:estimbarU}
We have $\bar U = m_+  - 2 \beta \eta + o (\beta \eta)$. In particular,
\begin{equation}\label{prop:ubar}
\bigg|\bar U - (m^+  -\beta \eta) \bigg| \le 2\beta \eta \, .
\end{equation}
\end{lem}
We now estimate the second term of the right hand side of \eqref{eq:presque}. 
\begin{lem}
\begin{equation}
f(U^{\eps,\beta}) - \psi_d f(\bar{U}) \le - \frac12 (1 -\psi_d) \beta \eta \, .
\label{eq:estim-f}
\end{equation}
\end{lem}
\begin{proof}
From Lemma~\ref{lem:estimbarU}, we have
\begin{eqnarray*}
U^{\eps,\beta} &=& \psi_d (m_+ - 2\beta \eta+o(\beta \eta))+ (1-\psi_d) (m_+ -\beta \eta) \\
& = & m_+ - \beta \eta - \psi_d (1 + o(1)) \beta \eta \, .
\end{eqnarray*}
In particular, $U^{\eps,\beta} < m_+$. Hence,
\begin{eqnarray*}
f(U^{\eps,\beta}) & = & f(U^{\eps,\beta})- f (m_+) + f(m_+) \\
& \le & -(f'(\bar{m}_+) +o(1) )(1+\psi_d+o(1)) \beta \eta  - \beta \eta \\
& \le & -(f'(\bar{m}_+)+1 +o(1) )  \beta \eta  
\end{eqnarray*}
Lemma~\ref{lem:estimbarU} also implies
$$
0 \le U^{\eps,\beta} - \bar U = (1-\psi_d) (1+o(1))\beta \eta \, .
$$
Hence, we obtain
\begin{eqnarray*}
  f(U^{\eps,\beta}) - \psi_d f(\bar{U}) &=& (1-\psi_d) f(U^{\eps,\beta}) + \psi_d (f(U^{\eps,\beta}) - f(\bar{U})) \\
  & \le & - (1-\psi_d) (f'(\bar{m}_+)+1 +o(1) ) \beta \eta  \\
&& + \psi_d (f' (\bar{m}_+) + o(1)) (1-\psi_d)(1 + o(1)) \beta \eta   \\
& \le & (-1+o(1)) (1-\psi_d)\beta \eta \le -\frac12 (1-\psi_d) \beta \eta \, .
\end{eqnarray*}
\end{proof}

We now turn to the third term of the right hand side of
\eqref{eq:presque} whose estimate is more delicate.  It is given by
the following technical lemma.
\begin{lem}\label{lem:ultime}
For any $\gamma_0 >0$, there exists a function $O(\eps^\alpha)$ such that
\begin{equation}\label{eq:ultime}
\psi_d \mathcal{I}^\eps_\alpha \bar{U}- \mathcal{I}^\eps_\alpha U^{\eps,\beta} 
\le  \gamma_0 \eta + O (\eps^\alpha) \; . 
\end{equation}
\end{lem}
\begin{proof}
We would like first to point out that we can forget the time variable in this proof
since it plays no role.

We first remark that   for $r_0 =\gamma/4$, 
$$
|z| \le \frac{r_0}\eps \Rightarrow
U^{\eps,\beta} (x+\eps z) = \psi_d (x+\eps z)(\bar{U}(x+\eps z) -m_++\beta \eta) + (m_+-\beta \eta) \, .
$$
Indeed, $|d(x+\eps z) - d(x)| \le r_0 = \gamma/4$ and this implies $d(x+\eps z) \in 
(\gamma/4,\gamma)$. In particular $d(x+\eps z ) \le \gamma$ and \eqref{prop:ubar} holds true.

We next approximate the quantity we are estimating by truncating large $z$'s. Precisely, 
using the previous remark and the fact that the mass of $J$ outside $B_{r_0/\eps}$ is $O(\eps^\alpha)$, 
we write
\begin{eqnarray*}
  \psi_d \mathcal{I}^\eps_\alpha \bar{U}- \mathcal{I}^\eps_\alpha U^{\eps,\beta} 
  &=& \psi_d \mathcal{J}^\eps \bar{U} - \mathcal{J}^\eps [ \psi_d (\bar{U} -m_++\beta \eta) 
  + (m_+-\beta \eta) ]
  +O (\eps^\alpha)\\
  &=&\psi_d \mathcal{J}^\eps (\bar{U} -m_++\beta\eta) 
  - \mathcal{J}^\eps  [ \psi_d (\bar{U} -m_++\beta \eta) ]  + O (\eps^\alpha)
\end{eqnarray*}
where the operator $\mathcal{J}^\eps$ is defined as follows
$$
\mathcal{J}^\eps \varphi (x) = \int_{\eps |z| \le r_0} [\varphi(x+\eps z) - \varphi(x) - D \varphi(x) \cdot 
\eps z \un_B( z) ] J(z)dz
$$
and where $O(\eps^\alpha)$ only depends on $\|Q\|_\infty$ and $\bar{m}_\pm$ (for $\eps$ small enough).

We next use the following equality 
\begin{multline*}
\mathcal{J}^\eps ( \varphi \varphi') (x) - \varphi(x) \mathcal{J}^\eps \varphi' (x) 
- \varphi'(x) \mathcal{J}^\eps \varphi (x) \\ 
= \int_{\eps |z| \le r_0} (\varphi(x+\eps z) - \varphi(x))
(\varphi'(x+\eps z )-\varphi'(x)) J(z) dz
\end{multline*}
with $\varphi = \psi_d$ and $\varphi'=\bar{U} -m_+ +\beta\eta$. We obtain
\begin{multline*}
\psi_d \mathcal{I}^\eps_\alpha \bar{U}- \mathcal{I}^\eps_\alpha U^{\eps,\beta} 
= - (\bar{U}  - m_+ +\beta \eta) \mathcal{J}^\eps \psi_{d}  \\
  - \int_{\eps |z| \le r_0} (\psi_d (x+\eps z) - \psi_d (x))(\bar{U}(x+\eps z) 
- \bar{U} (x)) J(z) dz + O (\eps^\alpha) \; . 
\end{multline*}
Recalling that \eqref{prop:ubar} holds true for $|z| \le \eps^{-1} r_0$, we write
\begin{multline*}
\psi_d \mathcal{I}^\eps_\alpha \bar{U}- \mathcal{I}^\eps_\alpha U^{\eps,\beta} 
\le 2 \beta \eta | \mathcal{J}^\eps \psi_d (x)| \\
+ \eps \| D \psi_d \|_{L^\infty(B(x,r_0))} 
\bigg(\eps \| D \bar{U} \|_{L^\infty(B(x,r_0))}  \bigg) \int_{|z| \le \eps}|z|^2 J (z) dz \\
+ \eps \| D \psi_d \|_{L^\infty(B(x,r_0))} 4 \beta \eta  \int_{\eps \le |z| \le r_0 \eps^{-1}}|z| 
  J(z) dz  + O (\eps^\alpha) \, .
\end{multline*}
We next estimate each term as follows.
\begin{eqnarray*}
| \mathcal{J}^\eps \psi_d (x)| &\le &C(\psi_d) \eps^{\alpha} \, ,\\
 \int_{|z| \le \eps}|z|^2 J (z) dz & \le & C \eps^{2-\alpha} \, , \\
 \int_{\eps \le |z| \le r_0 \eps^{-1}}|z|  J(z) dz & \le &
\left\{\begin{array}{ll} 
C \eps^{1-\alpha} & \text{ if } \alpha > 1 \\
C |\ln \eps| & \text{ if } \alpha = 1 
\end{array}
\right. \le \frac1\eps\, .
\end{eqnarray*}
The first estimate is easily obtained by adapting the arguments used above to estimate
$R^1[T_Q]$ and $R^2 [T_Q]$. 
Moreover, the constant $C(\psi_d)$ only depends on $\|\psi_d\|_\infty =1$ and 
$\|D^2 \psi_d \|_{L^\infty (B(x,r_0))}$. This last quantity only depends on 
$\gamma$ and $\|Dd\|_{L^\infty(B(x,r_0))}$, $\|D^2d\|_{L^\infty(B(x,r_0))}$ .
Hence, by using Lemma~\ref{lem:estim-nablaubar}, we have
$$
\psi_d \mathcal{I}^\eps_\alpha \bar{U}- \mathcal{I}^\eps_\alpha U^{\eps,\beta} 
 \le  C ( \eta \eps^{\alpha} +   \eps^{3-\alpha} +   \eta) + O (\eps^\alpha)
$$
(we used that $\beta \le 1$ for instance).
We achieve the proof by choosing $\eps$ small enough so that
$$
C  (\eta \eps^{\alpha} +   \eps^{3-\alpha} +  \eta) \le \gamma_0 \, .
$$
\end{proof}
We now combine \eqref{eq:presque}, \eqref{eq:estim-f} and \eqref{eq:ultime} to get
$$
\partial_t U^{\eps,\beta} + \frac1{\eps \eta} 
\left(f(U^{\eps,\beta}) - \mathcal{I}^\eps_\alpha U^{\eps,\beta} \right) 
\le  C \beta \eta  -  (1-\psi_d) \frac\beta{2\eps} - \psi_d \frac\beta{2\eps} 
+ \frac{\gamma_0}{ \eps} + o ( \frac1\eps ) \, .
$$
This is where it is convenient to choose $\psi$ such that $\psi(5\gamma/8+r) = 1 - \psi (5\gamma/8-r)$
since in this case, $\max(\psi_d,1-\psi_d) \ge 1/2$ and we obtain
\begin{eqnarray*}
\partial_t U^{\eps,\beta} + \frac1{\eps \eta} 
\left(f(U^{\eps,\beta}) - \mathcal{I}^\eps_\alpha U^{\eps,\beta} \right) 
\le  C \beta \eta  -   \frac\beta{4 \eps} 
+ \frac{\gamma_0}{ \eps} + o\left(\frac1\eps \right) \; .  
\end{eqnarray*}
Choosing now $\gamma_0$ small enough, we finally get
\begin{eqnarray*}
\partial_t U^{\eps,\beta} + \frac1{\eps \eta} 
\left(f(U^{\eps,\beta}) - \mathcal{I}^\eps_\alpha U^{\eps,\beta} \right) 
\le  C(\psi) \beta\eta  -  \frac\beta{8 \eps} \: .
\end{eqnarray*}
It is now clear that for $\eps$ small enough, $U^{\eps,\beta}$ is a subsolution of \eqref{eq:rd}
in $[t_0,t_0+h] \times \R^N$. 
\end{proof}

\section{Proof of Lemma~\ref{lem:a-eps-conv}}
\label{sec:lem}

This section is devoted to the study of the average of oscillating
terms. Their behaviour as $\eps \to 0$ was given by
Lemma~\ref{lem:a-eps-conv} whose proof was postponed. 
We first deal with the singular case with $\alpha >1$.
We next prove the result in the regular case.  
We then state the equivalent lemma for the case $\alpha <1$
since ideas will be used in the case $\alpha =1$.
We finally prove Lemma~\ref{lem:a-eps-conv} in the case $\alpha =1$.  

\begin{proof}[Proof of Lemma~\ref{lem:a-eps-conv} in the singular case for $\alpha >1$.]
We first recall the definition of $\bar{a}_\eps$ and $W$. For the sake of clarity,
we do not write $e$ and $h$ variables of $q$ since they play no role in the present argument.
\begin{multline*}
\bar{a}_\eps (e,t,x)  =  \int \dot{q} (\xi) a_\eps (\xi,e,t,x) d\xi \\
 =  \frac1\eps \iint \dot{q} (\xi) \bigg\{  q ( \xi + e \cdot z + \eps W (t,x,z))
-  q ( \xi + e \cdot z)  \bigg\} J(z) d\xi dz  
\end{multline*}
with
$$
W (t,x,z) = 
\frac1{\eps^2} [ d(t,x+\eps z) - d (t,x) - \eps D d(t,x) \cdot z ] \, .
$$
We proceed  in several steps. \medskip

\noindent \textbf{Step 1: reduction to the study of the 
singular integral around the origin for quadratic $W$'s.}
Let us choose $r_\eps$ such that we also have (see \eqref{cond1:J})
$$
\frac1\eps \int_{|z| \ge r_\eps} J (z) dz \to 0 \quad \text{ as } \eps \to 0\; .
$$
For instance we consider $r_\eps = \eps^{-\beta}$ with $\beta > 1/\alpha$. 
In view of Condition~\eqref{cond1:J}, we thus can assume from now on that
$$
J(z) = g (\hat z) \frac1{|z|^{N+\alpha}} \, .
$$
Since $q$ is bounded, it is therefore enough to study the convergence of 
\begin{multline*}
b_\eps (e,t,x)   \\
 = \int \int_{|z| \le r_\eps}  \dot{q} (\xi) 
\frac1\eps  \bigg\{  q (\xi + e \cdot z+ \eps W (t, x,z), e )
    -q ( \xi + e \cdot z)  \bigg\} J(z) d\xi dz  \; .
\end{multline*}

For $|z| \le r_\eps$,
$$
W(t,x,z) \to \frac12 D^2 d (t,x) z \cdot z \quad \text{ as } \eps \to 0  
$$
as soon as one chooses $r_\eps$ such that $\eps r_\eps \to 0$. Hence we take
$\beta \in (\alpha^{-1},1)$. If $B$ denotes $\frac12 D^2 d (t,x)$,
we have for $|z| \le r_\delta$,
$$
 |W(t,x,z) - B z \cdot z | \le \delta |z|^2 \, .
$$
By using the monotonicity of $q$, we  thus can reduce the study of $b_\eps$ to the study of 
$$
c_\eps = \int \int_{|z| \le r_\eps}  \dot{q} (\xi) \frac1\eps  \bigg\{  q (\xi + e\cdot z + \eps C z \cdot z)
    -q ( \xi + e \cdot z)  \bigg\} J(z) d\xi dz  \; .
$$
\medskip

\noindent \textbf{Step 2: integrating by parts.} By using a system of coordinates where $z_1 = e\cdot z$
and $z=(z_1,z')$, we can decompose the matrix $C$ as follows
$$
C = \left[ \begin{array}{ll} c_1& v^*\\ v& C'\end{array} \right]
$$
Hence, we can write
\begin{eqnarray*}
c_\eps & =&  \int \int_{|z| \le r_\eps}  \int_0^1 \dot{q} (\xi) \dot{q} (\xi + z_1+ \eps \tau C z \cdot z)
(Cz \cdot z) J(z_1,z') d\xi dz_1 dz' d\tau \\
 & =&  \int \int_{|z| \le r_\eps}  \int_0^1 \dot{q} (\xi) \dot{q} (\xi + z_1+ \eps \tau C z \cdot z)
(Cz \cdot z) J(z_1,z') d\xi dz_1 dz' d\tau \\
&=&\int \int_{|z| \le r_\eps}  \int_0^1 \dot{q} (\xi) \partial_{z_1} 
\bigg\{ q (\xi + z_1+ \eps \tau C z \cdot z) -q(\xi) \bigg\}  \\
&& \times
\frac{Cz \cdot z}{1 + 2 \eps \tau ( c_1 z_1 + v\cdot z')} J(z_1,z') d\xi dz_1 dz' d\tau \\
\end{eqnarray*}
We now integrate by parts with respect to $z_1$. 
\begin{eqnarray*}
c_\eps & = & - \int \int_{|z| \le r_\eps}  \int_0^1 \dot{q} (\xi)
\bigg\{ q (\xi + z_1+ \eps \tau C z \cdot z)\bigg\} \\
&& \quad \times  \partial_{z_1}\bigg\{ \frac{Cz \cdot z}{1 + 2 \eps \tau (c_1 z_1+ v \cdot z')} J(z_1,z') \bigg\} 
 d\xi dz_1 dz' d\tau + (BT)_+^1 - (BT)_-^1 
\end{eqnarray*}
with
\begin{multline*}
(BT)_\pm^1 = \int \int_{|z'| \le r_\eps}  \int_0^1 \dot{q} (\xi)
\bigg\{ q (\xi  \pm \sqrt{r_\eps^2 -|z'|^2} \\ + \eps \tau C (\pm \sqrt{r_\eps^2 -|z'|^2}, z')\cdot 
(\pm \sqrt{r_\eps^2 -|z'|^2}, z'))\bigg\} \\
\times  \bigg\{ \frac{ C (\pm \sqrt{r_\eps^2 -|z'|^2}, z')\cdot 
(\pm \sqrt{r_\eps^2 -|z'|^2}, z')}{1 + 2 \eps \tau ( \pm c_1\sqrt{r_\eps^2 -|z'|^2} + v \cdot z')} 
J(\pm \sqrt{r_\eps^2 -|z'|^2},z') \bigg\}  d\xi  dz' d\tau \, .
\end{multline*}
We next integrate by parts with respect to $\xi$. 
\begin{eqnarray*}
c_\eps & = & \int \int_{|z| \le r_\eps}  \int_0^1 q (\xi)
\bigg\{ \dot{q} (\xi + z_1+ \eps \tau C z \cdot z)\bigg\} \\
&&  \times  \partial_{z_1}\bigg\{ \frac{Cz \cdot z}{1 + 2 \eps \tau (c_1 z_1 + v\cdot z')} J(z_1,z') \bigg\} 
 d\xi dz_1 dz' d\tau \\ 
&& + (BT)_+^1 - (BT)_-^1 \\
& =&  \int \int_{|z| \le r_\eps}  \int_0^1 q (\xi)
\partial_{z_1} \bigg\{ q (\xi + z_1+ \eps \tau C z \cdot z) -q (\xi) \bigg\} \\
&&  \times  \frac1{1+2\eps \tau (c_1 z_1+v\cdot z')}
\partial_{z_1}\bigg\{ \frac{Cz \cdot z}{1 + 2 \eps \tau (c_1 z_1 + v\cdot z')} J(z_1,z') \bigg\} 
 d\xi dz_1 dz' d\tau \\
&& + (BT)_+^1 - (BT)_-^1 
\end{eqnarray*}
We finally integrate by parts in $z_1$ and we get
\begin{equation}\label{eq:cdeps}
c_\eps = d_\eps + (BT)_+^1 - (BT)_-^1 + (BT)_+^2 - (BT)_-^2 
\end{equation}
with 
\begin{multline*}
d_\eps  =  -\int \int_{|z| \le r_\eps}  \int_0^1 q (\xi)
\bigg\{ q (\xi + z_1+ \eps \tau C z \cdot z) -q (\xi) \bigg\} \\
 \times  \partial_{z_1}\bigg\{ \frac1{1+2\eps \tau (c_1 z_1+v\cdot z')} 
 \partial_{z_1}\bigg\{ \frac{Cz \cdot z}{1 + 2 \eps \tau (c_1 z_1 + v\cdot z')} J(z_1,z') \bigg\} \bigg\}
 d\xi dz_1 dz' d\tau 
\end{multline*}
\begin{eqnarray*}
(BT)^2_\pm &=& \int \int_{|z'| \le r_\eps}  \int_0^1  q (\xi)  
\bigg\{ q (\xi \pm \sqrt{r_\eps^2 - |z'|^2} \\
&& + \eps \tau (\pm\sqrt{r_\eps^2 -|z'|^2},z') \cdot (\pm\sqrt{r_\eps^2 -|z'|^2},z')) - q(\xi) \bigg\} \\
&& 
 \times \frac1{1 + 2 \eps \tau (\pm c_1\sqrt{r_\eps^2 -|z'|^2}+ v\cdot z')} \\
&&\partial_{z_1}\bigg\{ \frac{Cz \cdot z}{1 + 2 \eps \tau ( c_1 (\cdot)+ v \cdot z') } 
J(\cdot,z') \bigg\} (\pm \sqrt{r_\eps^2 -|z'|^2})  d\xi  dz' d\tau \, .
\end{eqnarray*}
We now study the limits of all terms in \eqref{eq:cdeps}.
\medskip

\noindent \textbf{Step 3: study of boundary terms.}
We start with $(BT)^1_\pm$. 
\begin{eqnarray*}
|(BT)^1_\pm| &\le& 2 (m_+(h) - m_-(h)) \| q\|_\infty \|g\|_\infty \int_{|z'| \le r_\eps}  
2 r_\eps^2 \frac{dz'}{r_\eps^{N+\alpha}} \\
 & \le & C r_\eps^{1-\alpha}   
\end{eqnarray*}
and this goes to $0$ as $\eps \to 0$. 

We now turn to $(BT)^2_\pm$. It is convenient to introduce the function 
$$
\Gamma (\tau,z_1) = (1 + 2 \eps \tau (c_1 z_1 + v\cdot z'))^{-1} \, .
$$
Since $\eps r_\eps \to 0$ as $\eps \to 0$, we deduce that for 
$\eps$ small enough, we have
\begin{eqnarray*}
| \Gamma (\tau,z_1) | & \le & 2  \\ 
| \partial_{z_1} \Gamma (\tau,z_1) | &\le& 8 |c_1| \eps \\
| \partial_{z_1} (\Gamma^2) (\tau,z_1) | & \le & 32 |c_1| \eps \\
|\partial^2_{z_1,z_1} \Gamma (\tau, z_1) | & \le & 64 |c_1|^2 \eps^2 \, .
\end{eqnarray*}
We next compute
\begin{multline*}
\partial_{z_1} (( Cz \cdot z)\Gamma (z_1) J (z_1,z')) 
\\ = 2 c_1 z_1 \Gamma J + (c_1z_1^2 + C' z' \cdot z') (\partial_{z_1} \Gamma) J   
+ (c_1 z_1^2 + C' z'\cdot z') \Gamma \partial_{z_1} J \,.
\end{multline*}
Now, since $J (z) = g(\hat{z}) |z|^{-N -\alpha}$, we deduce that
$$
|\partial_{z_1} J (z) | \le \frac{C}{|z|^{N+\alpha +1}} \, .
$$
We thus conclude that for $|z'| \le r_\eps$ and $z_1$ such that $|z|=r_\eps$, we have
$$
| \partial_{z_1} (( Cz \cdot z)\Gamma J )| \le C
\frac{r_\eps}{r_\eps^{N+\alpha}} + C \frac{\eps
  r_\eps^2}{r_\eps^{N+\alpha}} + C \frac{r_\eps^2}{r_\eps^{N+\alpha
    +1}} \le C r_\eps^{-N - \alpha}
$$
and we get 
$$
\int_{|z'| \le r_\eps, z_1^2 + |z'|^2 = r_\eps^2}| \partial_{z_1} (( Cz \cdot z)\Gamma (\tau,z_1) J (z_1,z'))| 
dz' \le C r_\eps^{-\alpha} \, .
$$
With this inequality in hand, we now derive 
\begin{eqnarray*}
G(-Cr_\eps)&=&- C r_\eps^{-\alpha} \int  (q(\xi) - q (\xi - Cr_\eps)) d\xi \\
&\le&  (BT)^2_\pm \le C r_\eps^{-\alpha} \int  (q (\xi + Cr_\eps)- q(\xi)) d\xi = G (C r_\eps) 
\end{eqnarray*}
where
$$
G(r) = \int (q(\xi+ r) -q (\xi)) d\xi \, . 
$$
It is clear  that $G$ is Lipschitz continuous and equals $0$ at $0$. Hence
$$
|(BT)^2_\pm | \le C r_\eps^{1-\alpha} \, .
$$
It thus goes to $0$ as $\eps \to 0$. 
\medskip

\noindent \textbf{Step 4: study of $d_\eps$.}
In order to study the main term $d_\eps$, we first write it as follows
$$
d_\eps = e_\eps + R_\eps 
$$
with 
\begin{eqnarray*}
e_\eps &=&  -\int \int_{|z| \le r_\eps}  \int_0^1 q (\xi)
\bigg\{ q (\xi + z_1+ \eps \tau C z \cdot z) -q (\xi)  \bigg\} \\ && \quad \times  
\Gamma^2 (\tau,z_1) \partial^2_{z_1,z_1}\bigg\{ (Cz \cdot z) J(z_1,z') \bigg\} 
 d\xi dz_1 dz' d\tau\\
R_\eps &=&  -\int \int_{|z| \le r_\eps}  \int_0^1 q (\xi)
\bigg\{ q (\xi + z_1+ \eps \tau C z \cdot z) -q (\xi)  \bigg\} \\ && \quad \times  
\bigg[ \bigg\{ (\partial_{z_1}\Gamma)^2 + \Gamma \partial_{z_1,z_1}^2 \Gamma \bigg\} 
\bigg\{ (Cz \cdot z) J(z_1,z') \bigg\} \\
&& \quad + \partial_{z_1} (\Gamma^2)
\partial_{z_1} \bigg\{ (Cz \cdot z) J(z_1,z') \bigg\} \bigg] d\xi dz_1 dz' d\tau \, .
\end{eqnarray*}
Let us now prove that $R_\eps$ goes to $0$ as $\eps \to 0$. We proceed as we did with $(BT)_\pm^2$. 
We first estimate the quantity $[ \dots ]$ in the definition of $R_\eps$. We use the estimates on $\Gamma$ 
and its derivatives, together with the estimate of $\partial_{z_1} J$. We obtain for 
$|z| \le r_\eps$, 
\begin{eqnarray*}
| [ \dots ] | & \le & C \frac{\eps^2}{|z|^{N+\alpha -2}} + C \frac{\eps}{|z|^{N+\alpha -1}} \le \frac{C\eps}{|z|^{N+\alpha-1}}
\end{eqnarray*}
since $\eps |z| \le \eps r_\eps \to 0$. By arguing as for $(BT)_\pm^2$, we conclude that
$$
|R_\eps | \le C \eps G ( C r_\eps) \int_{|z| \le r_\eps} |z|^{-N-\alpha +1} dz \le C (\eps r_\eps) r_\eps^{1-\alpha} 
$$
and the right hand side of the previous inequality goes to $0$ as $\eps \to 0$. 

It remains to study the limit of $e_\eps$. By dominated convergence theorem, we obtain that
it converges towards
\begin{eqnarray*}
e_0 
& =& -\int_{\R_\xi} q(\xi) \int_{\R_{z_1}} (q(\xi + z_1) - q(\xi)) \\
&  & \times \partial_{z_1,z_1}^2 \left\{ \int_{z'} (Cz \cdot z) 
g (\hat{z}) \frac{dz'}{(|z_1^2 + |z'|^2)^{(N+\alpha) /2}}\right\}
d\xi dz_1 \\ 
& =& -\int_{\R_\xi} q(\xi) \int_{\R_{z_1}} (q(\xi + z_1) - q(\xi)) 
 \partial_{z_1,z_1}^2 \left\{ \frac{|z_1|^{1-\alpha}}{\alpha(\alpha -1)}\mathrm{Tr} (A_g (e) C) \right\} d\xi dz_1  \\
& =& -  \int_{\R_\xi} q(\xi) \int_{\R_{z_1}} (q(\xi + z_1) - q(\xi)) \frac{dz_1}{|z_1|^\alpha }d \xi
 \bigg( \mathrm{Tr} (A_g (e) C) \bigg) \\
& = & \iint (q(\xi+ z_1) - q(\xi))^2 \frac{dz_1 d\xi}{|z_1|^{1+\alpha}} 
\bigg( \mathrm{Tr} (A_g (e) C) \bigg)
\end{eqnarray*}
where $A_g (e)$ is defined in \eqref{def:ag}.  The proof is now complete. 
\end{proof}
\begin{proof}[Proof of Lemma~\ref{lem:a-eps-conv} in the regular case.]
The proof of the lemma in this case is divided into two steps. 

\noindent \textbf{Step 1: reduction to the study of the 
singular integral around the origin for quadratic $W$'s.}
As explained in the proof of Lemma~\ref{lem:a-eps-conv} for $J (z) = g(\hat{z}) |z|^{-N-\alpha}$ and $\alpha >1$,
it is enough to study the convergence of 
\begin{multline*}
b_\eps (e,t,x)  =  \int \int_{|z| \le r_\eps} 
\dot{q} (\xi) a_\eps (\xi,e,t,x) d\xi dz \\
 = \int \int_{|z| \le r_\eps}  \dot{q} (\xi) 
\frac1\eps  \bigg\{  q (\xi + e \cdot z+ \eps W (t, x,z), e )
    -q ( \xi + e \cdot z)  \bigg\} J(z) d\xi dz  
\end{multline*}
where we recall that $r_\eps = \eps^{-\beta}$ with $\beta > 1/\alpha$. 

Remark that there exists $C_R>0$ such that for any $(t,x) \in B_R$ and $z \in B$, 
\begin{equation}\label{eq:estimW}
|W(t,x,z)| \le C_W |z|^2  
\end{equation}
and that, for $|z| \le r_\eps$,
$$
W(t,x,z) \to \frac12 D^2 d (t,x) z \cdot z \quad \text{ as } \eps \to 0  
$$
as soon as one chooses $r_\eps$ such that $\eps r_\eps \to 0$. 
We conclude that the integrand of $b_\eps$ converges towards
$$ 
\frac12 \dot{q^0} (\xi) \dot{q^0} ( \xi + e \cdot z) ( D^2 d (t,x) z\cdot z ) \; J(z) \, .
$$
This explains why we expect the limit of $\bar{a}_\eps$ to be given by \eqref{eq:matrix-rd>1bis}.

We can apply dominated convergence theorem outside the unit ball $B$. Hence, we reduce
the study of the limit of $b_\eps$ to the one of
\begin{multline*}
c_\eps (e,t,x) \\ =  \int \int_{|z| \le 1} \dot{q} (\xi) 
\frac1\eps  \bigg\{  q (\xi + e \cdot z+ \eps W (t, x,z), e )
    -q ( \xi + e \cdot z)  \bigg\} J(z) d\xi dz  \; .
\end{multline*}
In order to do so, we introduce
\begin{eqnarray*} 
c^\pm_\eps  &=&\int \int_{|z| \le 1}  \dot{q} (\xi) 
\frac1\eps  \bigg\{  q (\xi + e \cdot z \pm \eps C_W |z|^2, e )
    -q ( \xi + e \cdot z)  \bigg\} J(z) d\xi dz  \; .
\end{eqnarray*}
We know from \eqref{eq:estimW} and the monotonicity property of $q$ that 
\begin{equation}\label{eq:gendarme}
c^-_\eps \le c_\eps  \le c^+_\eps 
\end{equation}
and it is thus enough to prove that integrals $c_\eps^\pm$ have limits
that are uniform with respect to $e,t,x$ to conclude. 
\medskip

\noindent \textbf{Step 2: integrating by parts.}
Recall that $|e|=1$ and let $z_1$ denote $e \cdot z$ and $z=(z_1,z')$. 
We now write
\begin{multline*}
c^+_\eps (e) \\= \int \int_{|z| \le 1}  \dot{q} (\xi) 
\frac1\eps  \bigg\{  q (\xi + z_1 + \eps C_W z_1^2 + \eps C_W |z'|^2, e )
    -q ( \xi+ z_1)  \bigg\} J(z) d\xi dz  \\
 =   C_W \int \int_{|z| \le 1} \int_0^1 \dot{q} (\xi)
\dot{q}  (\xi+ z_1 + \eps \tau C_W z_1^2 + \eps \tau C_W |z'|^2, e ) \; |z|^2 
J(z) \; d\xi dz d\tau \\
 =  C_W \int  \int_0^1 \int_{|z| \le 1} \dot{q} (\xi)
\partial_{z_1} \bigg (q  (\xi+ z_1 + \eps \tau C_W z_1^2 + \eps \tau C_W |z'|^2, e ) \bigg) \\
\times 
\frac{|z|^2 J(z)}{1 + 2 \eps \tau C_W z_1} \; d\xi d\tau dz' dz_1 \, .
\end{multline*}
We next integrate by parts and get 
$$
c^+_\eps  = d_\eps^+ + (BT)_+ + (BT)_-  \, .
$$
with 
\begin{eqnarray*}
d_\eps^+ 
 &=& -  C_W \int  \int_0^1 \int_{|z| \le 1} \dot{q} (\xi)
 q  (\xi+ z_1 + \eps \tau C_W z_1 + \eps \tau C_W |z'|^2, e ) \\
&& \times K (\tau,z) \; d\xi d\tau dz' dz_1  \\
(BT)_\pm & = & \pm C_W \int  \int_0^1 \int_{|z| \le 1} \dot{q} (\xi)
 q  (\xi \pm \sqrt{1 - |z'|^2} + \eps \tau C_W , e )  \\
&& \times
\frac{ J (\pm \sqrt{1 - |z'|^2}, z')}{1 \pm 2 \eps \tau C_W \sqrt{1 - |z'|^2}} \\
\end{eqnarray*}
where 
\begin{eqnarray*}
K (\tau, z) &=& \partial_{z_1} \bigg( \frac{|z|^2 J(z)}{1 + 2 \eps \tau C_W z_1} \bigg) \\
& = &  \frac{\partial_{z_1} (|z|^2 J(z))}{1 + 2 \eps \tau C_W z_1} 
- 2 \eps \tau C_W \frac{|z|^2 J(z)}{(1 + 2 \eps \tau C_W z_1)^2} \, . 
\end{eqnarray*}
After remarking that for $J \in L^1 \cap C^0_c$, we have
$$
| K (\tau,z )| \le C ( |z|^2 J (z) + |\nabla ( |z|^2 J (z) ) | ) \in L^1 (B) \, ,
$$
it is clear that we can apply dominated convergence theorem in each integral. 
The proof is now complete. 
\end{proof}
In the case $\alpha <1$, the ansatz used to treat the case $\alpha \ge
1$ yield oscillating terms with the following form. 
\begin{eqnarray*}
  a_\eps (\zeta,s,y) & =& \frac1{\eps^\alpha} \int \bigg( 
  q (\zeta + \frac{\phi(s,y+\eps z)- \phi (s,y)}\eps )-q (\zeta + e
  \cdot z) \bigg) \frac{dz}{|z|^{N+\alpha}} \, .
\end{eqnarray*}
Their average is thus defined as follows
$$
\bar{a}_\eps (s,y) =  \int a_\eps(r,s,y) \dot{q} (r) dr \, .
$$
Even if we are not able (yet!) to treat the case $\alpha <1$, we
think this can be of interest to explain what is the limit of the
average as $\eps \to 0$ in order to justify our conjecture about
the limit we expect in the case $\alpha <1$. Another reason for
including such a result is that its proof shares ideas with the
one corresponding one for the case $\alpha =1$.
\begin{lem}[Uniform convergence of approximate coefficients (II)] 
\label{lem:a-eps-conv-2}
Consider a smooth function $d: (0,+\infty) \times \R^N \to \R$ such 
that $|Dd (t,x)|=1$. Then, as $\eps \to 0$,
$$
\bar{a}_\eps (e,t,x) \to \kappa [x,d(t,\cdot)] 
$$
and the limit is uniform with respect to $(e,t,x) \in \S \times  Q_\gamma$. 
\end{lem}
\begin{proof}[Proof of Lemma~\ref{lem:a-eps-conv-2}]
In the case $\alpha <1$, we first make a change of variables as follows
\begin{multline*}
\bar{a}_\eps (e,t,x)  =  \int \dot{q} (\xi) a_\eps (\xi,e,t,x) d\xi \\
= 
\frac1{\eps^\alpha} \iint \dot{q} (\xi) \bigg\{  q (\xi + \frac{d(t,x+\eps z)-d(t,x)}\eps , e )
-  q ( \xi + e \cdot z)  \bigg\} J(z) d\xi dz \\
 =  \iint \dot{q} (\xi) \bigg\{  q (\xi + \frac{ d(t,x +\bar{z})- d(t,x)}\eps  , e )
-  q ( \xi + \frac{e \cdot \bar{z}}\eps)  \bigg\} J_\eps(\bar{z}) d\xi d \bar{z} 
\end{multline*}
with
$$
J_\eps (\bar{z}) = \frac1{\eps^{N+\alpha}} J \left( \frac{\bar{z}}\eps \right) \, .
$$

\noindent \textbf{Step 1: reduction to the study of the singular integral around the origin
for quadratic $W$'s.}
Remark next that it is easy to pass to the limit in the integrand; indeed,
\begin{multline*}
 q (\xi + \frac{ d(t,x +\bar{z})- d(t,x)}\eps  , e )
-  q ( \xi + \frac{e \cdot \bar{z}}\eps) 
 \\ \to (\bar{m}_+ - \bar{m}_-) 
\left( \un_{(d (t,x+\cdot) > d(t,x), e \cdot (\cdot ) < 0)}
- \un_{(d (t,x+\cdot) < d(t,x), e \cdot (\cdot ) > 0)} \right)
\end{multline*}
and 
$$
J_\eps (z) \to g (\frac{z}{|z|}) \frac1{|z|^{N+\alpha}} \, .
$$
The difficulty is to deal with the singular measure. Hence, it is enough to study, 
as in the case $\alpha =1$, 
\begin{eqnarray*}
c_\eps^+  &=&  \int_{\R_\xi} \int_{| \bar{z}| \le \delta}  
\dot{q} (\xi) \bigg\{  q (\xi + \frac{z_1 + C_W z_1^2 + C_W |z'|^2}\eps  , e )
-  q ( \xi + \frac{z_1}\eps)  \bigg\} 
\\ && \times J_\eps(z_1,z') d\xi dz' dz_1  \\
& = & C_W \int_{\R_\xi} \int_{|\bar z| \le \delta} \int_0^1 
\dot{q} (\xi) \dot{q} (\xi + \frac{z_1 + C_W \tau z_1^2 + C_W \tau |z'|^2}\eps  , e )
\\ && \times\frac{ |z|^2 }\eps J_\eps (z_1,z')  d\xi dz d\tau \\
& = & C_W \int_{\R_\xi} \int_{|\bar z| \le \delta} \int_0^1 
\dot{q} (\xi) \partial_{z_1} 
\left( q (\xi + \frac{z_1 + C_W \tau z_1^2 + C_W \tau |z'|^2}\eps  , e ) \right)
\\ && \times \frac{ |z|^2 }{1 + 2 C_W \tau z_1} J_\eps (z)  d\xi dz d\tau 
\end{eqnarray*}
where $\bar{z} = z_1 e + z'$.
\medskip

\noindent \textbf{Step 2: integrating by parts.} By integrating by parts, we obtain
$$
c_\eps^+ = d_\eps^+ + (BT)_+ + (BT)_- 
$$
where
\begin{eqnarray*}
d_\eps^+ & =&  - C_W \int_{\R_\xi} \int_{|\bar z| \le \delta} \int_0^1 
\dot{q} (\xi) q (\xi + \frac{z_1 + C_W \tau z_1^2 + C_W \tau |z'|^2}\eps  , e )
\\ && \times L_\eps (\tau,z)
 d\xi dz d\tau \\
(BT)_\pm & = &  \pm C_W \int_{\R_\xi} \int_{|z'| \le \delta} \int_0^1 
\dot{q} (\xi) q (\xi + \frac{\pm \sqrt{\delta^2 -|z'|^2} + C_W \tau \delta^2}\eps  , e ) 
\\ && \times\delta^2 
J_\eps ( \pm \sqrt{\delta^2 -|z'|^2},z') dz'  d\xi d\tau \\
\end{eqnarray*}
with
\begin{eqnarray*}
L_\eps (\tau, z)  &=& \partial_{z_1} \left( \frac{ |z|^2 J_\eps (z) }{1 + 2 C_W \tau z_1} \right) \\
& =&  \frac{ \partial_{z_1} (|z|^2 J_\eps (z) )}{1 + 2 C_W \tau z_1}  
- 2 C_W \tau \frac{|z|^2 J_\eps (z) }{(1 + 2 C_W \tau z_1)^2} \, .
\end{eqnarray*}
Condition~\eqref{cond2:J} ensures that 
$$
| L_\eps (\tau,z) | \le C K (z) \in L^1 (B) 
$$
and dominated convergence can be used to prove the convergence of $d_\eps^+$.
As far as boundary terms are concerned, we use \eqref{cond1:J} to get a constant $C>0$ such that
$$
J_\eps (z) \le \frac{C}{|z|^{N+\alpha}}
$$
and this implies
$$
J_\eps  ( \pm \sqrt{\delta^2 -|z'|^2}, z') \le \frac{C}{\delta^{N+\alpha}} \, .
$$
Hence, dominated convergence can be applied to boundary terms too. The proof is now complete. 
\end{proof}
\begin{proof}[Proof of Lemma~\ref{lem:a-eps-conv} in the singular case for $\alpha =1$.]
The proof is divided in several steps. 

\noindent \textbf{Step 1: reduction to the study of the singular integral around the origin
for quadratic $W$'s.} Let us fix $\delta >0$. There exists $r_\delta >0$ such 
that for any $\bar{z} \in B_{r_\delta}$, 
$$
| d(t,x+\bar{z}) - d(t,x) - Dd(t,x) \cdot \bar{z} - \frac12 D^2 d(t,x) \bar{z} \cdot \bar{z}| \le \delta|\bar{z}|^2 \; . 
$$
Consequently, for $z$ such that $|z| \le  r_\delta \eps^{-1}=:R_\eps$, we have
\begin{equation}\label{eq:estimWbis}
|W(t,x,z) - \frac12 Bz \cdot z | \le \delta |z|^2 \; .
\end{equation}
We remark next that 
\begin{eqnarray*}
\int_{|z| \le 1 }  |z|^2 J(z) dz &=&O(1) = o ( |\ln \eps|) \, , \\
\int_{|z| \ge R_\eps} J (z) dz &=& O (\eps) = o (\eps |\ln \eps|) \, .
\end{eqnarray*}
It is therefore enough to study the convergence of
\begin{multline*}
b_\eps \\ = \frac{1}{\eps |\ln \eps |} 
\iint_{1 \le |z| \le R_\eps} \dot{q} (\xi) [q (\xi + e \cdot z +\eps W(t,x,z)) - q (\xi + e\cdot z)] J(z) 
dz d\xi \; .
\end{multline*}
By using \eqref{eq:estimWbis} together with the monotonicity of $q$, it is even enough to study the 
convergence of 
\begin{multline*}
c_\eps   = \\ \frac{1}{\eps |\ln \eps |} 
\iint_{1 \le |z| \le R_\eps} 
\dot{q} (\xi) [q (\xi + z_1 + \eps c_1 z_1^2 + \eps C' z' \cdot z') - q (\xi + e\cdot z)] J(z) dz d\xi 
\end{multline*}
for any $(N-1) \times (N-1)$ symmetric matrix $C$ and any constant $c_1 \in \R$. 
Precisely, we next prove that, as $\eps \to 0$, 
$$ 
c_\eps  \to 2 (\bar{m}_+-\bar{m}_-)^2 \int_{\mathbb{S}^{N-2}}
g(\theta) \; C'\theta \cdot \theta \;  \sigma (d\theta) 
$$
where $g$ appears in \eqref{cond1:J}. 
\medskip

\noindent \textbf{Step 2: change of variables and domain decomposition.}
We now introduce the function $F_q$ defined as follows
$$
F_q (a) = \int \dot{q} (\xi) q (\xi + a) d\xi \, .
$$
We remark that $F_q$ is bounded, non-decreasing, Lipschitz continuous and sastisfies
$$
F_q (a) \to \left\{ \begin{array}{ll} \bar{m}_+ (\bar{m}_+ - \bar{m}_-) 
& \text{as } a \to + \infty \\ \bar{m}_-  (\bar{m}_+ - \bar{m}_-) &
\text{as } a \to - \infty \end{array} \right.
$$
and 
$$
\| F_q ' \|_\infty \le \left(\int \dot{q}^2 (\xi) d\xi\right)^{1/2} \, .
$$
We now rewrite $c_\eps$ with this new function
$$
c_\eps = \frac{1}{\eps |\ln \eps | } \int_{1 \le |z| \le R_\eps} 
\bigg\{ F_q (z_1 + \eps c_1 z_1^2 + \eps C' z' \cdot z') - F_q (z_1) \bigg\} J (z) dz
$$
Through the change of variables $z' = \eps^{-1} r \theta$ and $z_1 = \eps^{-1} r t_1$, we get
\begin{multline*}
c_\eps = \frac1{|\ln \eps |} \int_{\theta \in \mathbb{S}^{N-2}} d \sigma (\theta)
\int_0^{r_\delta} \frac{dr}{r^2} \int_{ \eps^2 r^{-2} \le 1 + t_1^2 \le r_\delta^2 r^{-2}} dt_1 \\
\times \bigg\{ F_q \left( \frac{r}\eps (t_1 + r c_1 t_1^2 + r C'\theta \cdot \theta) \right) 
- F_q \left( \frac{r}\eps t_1 \right)\bigg\}  
g \left( \frac{(t_1, \theta)}{\sqrt{t_1^2 + 1}} \right)  \frac1{(1+t_1^2)^{(N+1)/2}} \, . 
\end{multline*}
Without loss of generality, we assumed here  for simplicity
that $$J (z) = g (z /|z|)|z|^{-N-1}$$ (see \eqref{cond2:J}). 
We next cut it into two pieces as follows
$$
c_\eps = d_\eps + R_\eps^1
$$
with 
\begin{eqnarray*}
d_\eps & = &\frac1{|\ln \eps |} \int_{\theta \in \mathbb{S}^{N-2}}  d \sigma (\theta)
\int_{\Theta \eps}^{r_\delta} \frac{dr}{r^2} \int_{ \eps^2 r^{-2} \le 1 + t_1^2 \le r_\delta^2 r^{-2}} dt_1 \\
&& \times \bigg\{ F_q (\dots) - F_q ( \dots)\bigg\} g (\dots)  \frac1{(1+t_1^2)^{(N+1)/2}} \, .\\
R_\eps^1 & = &\frac1{ |\ln \eps |} \int_{\theta \in \mathbb{S}^{N-2}}  d \sigma (\theta)
\int_0^{\Theta \eps} \frac{dr}{r^2} \int_{ \eps^2 r^{-2} \le 1 + t_1^2 \le r_\delta^2 r^{-2}} dt_1 \\
&& \times \bigg\{ F_q (\dots) - F_q ( \dots)\bigg\} g (\dots)  \frac1{(1+t_1^2)^{(N+1)/2}} 
\end{eqnarray*}
for some $\Theta >0$ to be fixed later. 
\medskip

\noindent \textbf{Step 3: study of $R_\eps^1$.}
We now use the fact that $F_q$ is Lispchitz continuous in order to get,
for $\Theta$ large enough,
\begin{eqnarray*}
| R_\eps^1 | & \le & \| g \|_{\infty} \| F_q'\|_\infty (|c_1| + |C'|)
|\mathbb{S}^{N-2}| \frac1{ |\ln \eps|}
\int_0^{\Theta \eps} dr  \\ 
& & \times \int_{ \eps^2 r^{-2} \le 1 + t_1^2 \le r_\delta^2 r^{-2}} dt_1 
\frac1\eps \frac1{(1+t_1^2)^{(N-1)/2}} \\
& \le & C \frac1{\eps |\ln \eps|}\int_0^{\Theta \eps} dr 
 \int_{ 1 \le |t_1| \le r_\delta r^{-1 }} dt_1 \, t_1^{-N+1} \\
& \le & C_\delta \frac1{\eps |\ln \eps|} (\eps - \eps^{N-1} ) \le \frac{C_\delta}{|\ln \eps|} \, .
\end{eqnarray*}
Hence
$$
R_\eps^1 \to 0 \text{ as } \eps \to 0 \, .
$$

\noindent \textbf{Step 4: study of $d_\eps$.}
First, we rewrite $d_\eps$ (for $\Theta \ge 1$ large enough) as follows
\begin{multline*}
d_\eps  = \frac1{|\ln \eps |} \int_{\theta \in \mathbb{S}^{N-2}}  d \sigma (\theta)
\int_{\Theta \eps}^{r_\delta} \frac{dr}{r^2} \\ \times \int_{ |t_1| \le \sqrt{r_\delta^2 r^{-2} -1}}
\bigg\{ F_q (\dots) - F_q ( \dots)\bigg\} g (\dots)  \frac1{(1+t_1^2)^{(N+1)/2}}  \, .
\end{multline*}
We now study the limit of the integrand of $d_\eps$. In view of the limits of $F_q(a)$ when 
$a \to \pm \infty$, this limit is not $0$ if $t_1$ satisfies
$$
\left\{ \begin{array}{l} t_1 \ge 0 \\ t_1 + c_1 r t_1^2 + r c_\theta \le 0 \end{array} \right.
\quad \text{ or } \quad 
\left\{ \begin{array}{l} t_1 \le 0 \\ t_1 + c_1 r t_1^2 + r c_\theta \ge 0 \end{array} \right.
$$
where $c_\theta = C' \theta \cdot \theta$. This is equivalent to
$$
0  \le t_1 \le \mathcal{T}_1 (r) 
\quad \text{ or } \quad 
\left\{ \begin{array}{l} t_1 \le -\frac1{2c_1 r} (1 + \sqrt{1 - 4 c_1 c_\theta r^2}) \simeq - \frac1{c_1 r}  \\
\text{or } -c_\theta r \le t_1 \le 0 
 \end{array} \right.
$$
with
$$
\mathcal{T}_1 (r) = \frac1{2 c_1 r} (\sqrt{1 - 4c_1 c_\theta r^2}- 1)  \simeq - c_\theta r \, .
$$
We use now the fact that $|t_1| \le \sqrt{r_\delta^2 r^{-2} -1}$ to get 
$$
0  \le t_1 \le \mathcal{T}_1 (r) \quad \text{ or } \quad  \mathcal{T}_1 (r) \le t_1 \le 0  \, .
$$
We next cut $d_\eps$ into pieces as follows
$$
d_\eps = e_\eps^+ + e_\eps^- + R_\eps^2
$$
with
\begin{eqnarray*}
e_\eps^\pm  &=& \frac1{|\ln \eps |} \int_{\theta \in \mathbb{S}^{N-2}}  d \sigma (\theta)
\int_{\Theta \eps}^{r_\delta} \frac{dr}{r^2} \int_{ 0 \le \pm t_1 \le |\mathcal{T}_1 (r)|} \\
&& \times
\bigg\{ F_q (\dots) - F_q ( \dots)\bigg\} g (\dots)  \frac1{(1+t_1^2)^{(N+1)/2}} \\
R_\eps^2 &=&\frac1{|\ln \eps |} \int_{\theta \in \mathbb{S}^{N-2}}  d \sigma (\theta)
\int_{\Theta \eps}^{r_\delta} \frac{dr}{r^2} \int_{ |\mathcal{T}_1 (r)| \le  |t_1| \le \sqrt{r_\delta^2 r^{-2} -1}} \\
&& \times \bigg\{ F_q (\dots) - F_q ( \dots)\bigg\}  g (\dots)  \frac1{(1+t_1^2)^{(N+1)/2}} \, .
\end{eqnarray*}

\noindent \textbf{Step 5: study of $e_\eps^\pm$.}
Since $r \le r_\delta$ and $r_\delta$ very small, we can replace $e_\eps^\pm$ with
\begin{eqnarray*}
f_\eps^\pm  &=& \frac1{|\ln \eps |} \int_{\theta \in \mathbb{S}^{N-2}}  d \sigma (\theta)
\int_{\Theta \eps}^{r_\delta} \frac{dr}{r^2} \int_{ 0 \le \pm t_1 \le |\mathcal{T}_1 (r)| } dz_1 \\
&& \quad \times \bigg\{ F_q \left( \frac{r}\eps (t_1 + r c_1 t_1^2 + r C'\theta \cdot \theta) \right)  
- F_q \left( \frac{r}\eps t_1  \right)\bigg\} g (0, \theta) \\
 &=& \frac1{|\ln \eps |} \int_{\theta \in \mathbb{S}^{N-2}}  d \sigma (\theta)
\int_{\Theta \eps}^{r_\delta} \frac{dr}{r} \int_{ 0 \le \pm z_1 \le  r^{-1}  |\mathcal{T}_1 (r)| } dz_1 \\
&& \quad \times  \bigg\{ F_q \left( \frac1\eps (z_1 + r c_1 z_1^2 + r^2 C'\theta \cdot \theta) \right)  
- F_q \left( \frac{z_1}\eps   \right)\bigg\} g (0, \theta)
\end{eqnarray*}
Now it is easy to conclude that 
\begin{multline*}
f_\eps^\pm \to \int_{\theta \in \mathbb{S}^{N-2}} ((c_\theta)^+ + (c_\theta)^-) g(0,\theta) 
d \sigma (\theta) (F_{q^0} (+\infty) - F_{q^0} (-\infty)) \\
= (\bar{m}_+ - \bar{m}_-)^2 \int_{\theta \in \mathbb{S}^{N-2}} (C \theta \cdot \theta) g(0,\theta) 
d \sigma (\theta) \, .
\end{multline*}
It remains to prove that $R^2_\eps$ converges towards $0$. 
\medskip

\noindent \textbf{Step 6: study of $R^2_\eps$.}
By the study we made in  Step~4, we conclude that the integrand of $R^2_\eps$ converges
towards $0$. To dominate convergence, we simply write
\begin{eqnarray*}
|R^2_\eps| &\le &\frac1{|\ln \eps|} |\mathbb{S}^{N-2}| \int_{\Theta \eps}^{r_\delta} \frac{dr}r \int_{ |z_1| \ge |c_\theta| /2 }
2 \|F_q\|_\infty \frac1{(1+t_1^2)^{(N+1)/2}} \\
& \le &
|\mathbb{S}^{N-2}| \frac{\ln (r_\delta) + |\ln \eps |}{|\ln \eps |} \int_{ |z_1| \ge |c_\theta| /2 }
2 \|F_q\|_\infty \frac1{(1+t_1^2)^{(N+1)/2}} \, .
\end{eqnarray*}
The proof of the lemma is now complete.
\end{proof}

\appendix

\section{Link between \eqref{eq:matrix-rd>1} and \eqref{eq:matrix-rd>1bis} }

In this section, we  explain the link between the two first formulae in 
\eqref{eq:matrix-rd>1} and \eqref{eq:matrix-rd>1bis} by giving a formal argument.

If $z_1$ denotes $e \cdot z$ and $z' = z - (e\cdot z)e \in \R^{N-1}$, then
this equation can be written as follows
\begin{multline*}
A (e) = \\ \int_{\R_\xi} d\xi \dot{q^0} (\xi) \int_{\R_{z_1}} dz_1 \dot{q^0} (\xi + z_1) 
\int_{z' \in \R^{N-1}} dz' \left[ \begin{array}{ll} z_1^2 & z_1 z'\\ z_1 (z')^* & z' \otimes z' \end{array} \right] J(z_1,z')  
\end{multline*}
We remark next that 
$$
\int_{z' \in \R^{N-1}} \left[ \begin{array}{ll} z_1^2 & z_1 z'\\ z_1 (z')^* & z' \otimes z' \end{array} \right] J(z_1,z')   dz'
= \frac1{\alpha(\alpha-1)}A_g (e) |z_1|^{1-\alpha} \, .
$$
Hence, $A (e) = K A_g (e)$ with
\begin{equation}\label{def:ji}
K  =   \frac1{\alpha(\alpha-1)}\int_{\R_\xi} d\xi \dot{q^0} (\xi) \int_{\R_{z_1}} dz_1 
\dot{q^0} (\xi + z_1) |z_1|^{1 -\alpha} \, .
\end{equation}
By integrating by parts in $z_1$ and $\xi$ and $z_1$ successively, we obtain
\begin{eqnarray*}
J_i & = & \frac1\alpha \int_{\R_\xi} d\xi \dot{q^0} (\xi) \int_{\R_{z_1}} dz_1 
( q^0 (\xi + z_1) -q^0 (\xi))|z_1|^{-(1+\alpha)} z_1 \\
& =& -\frac1\alpha \int_{\R_\xi} d\xi q^0 (\xi) \int_{\R_{z_1}} dz_1 
( \dot{q^0} (\xi + z_1) - \dot{q^0} (\xi) ) |z_1|^{-(1+\alpha)} z_1 \\
& =&  -\int_{\R_\xi} d\xi q^0 (\xi)  \int_{\R_{z_1}} dz_1 
( q^0 (\xi + z_1)- q^0 (\xi) - \dot{q^0} (\xi) z_1  ) |z_1|^{-\alpha} \\
& =&  \int  q^0 (\xi) (-\Delta)^{(\alpha-1) /2} [q^0] (\xi) d \xi \\
 & = &  \int \int (q^0 (\xi+z_1) - q^0 (\xi))^2 \frac{dz_1}{|z_1|^\alpha} d\xi \, .
\end{eqnarray*}

\paragraph{Acknowledgments.} The first author was partially supported by ANR project ``MICA'' (ANR-06-BLAN-0082)
 and ``EVOL'' (). 

\bibliographystyle{siam}

\end{document}